\documentclass[12pt]{amsart}

\usepackage{tabmac,amsthm,amsmath,times,amsopn,amsfonts,amssymb,epsfig,anysize,mathtools}
\usepackage{graphicx}
\usepackage[all]{xy}
\usepackage{pifont,txfonts,graphicx,xcolor, verbatim, extarrows}
\usepackage{pb-diagram}
\usepackage{tikz-cd}
 \usepackage{pdfpages}
 \usepackage{tikz}
\usetikzlibrary{arrows, patterns}

\usepackage[margin=1in,a4paper]{geometry}

\newtheorem{theorem}{Theorem}
\newtheorem*{mainthm}{Main Theorem}
\newtheorem{lemma}[theorem]{Lemma}
\newtheorem{proposition}[theorem]{Proposition}
\newtheorem{prop}[theorem]{Proposition}

\newtheorem{definition}[theorem]{Definition}

\theoremstyle{definition}
\newtheorem{example}[theorem]{Example}

\theoremstyle{remark}
\newtheorem{remark}[theorem]{Remark}

\newcommand{\diag}{\mathrm{diag}}

\newcommand{\cyc}{\phi}

\newcommand{\Gr}{\mathrm{Gr}}
\newcommand{\Fl}{\mathrm{Fl}}
\newcommand{\AfGr}{\mathcal{G}r}

\newcommand{\QH}{\mathrm{QH}}
\newcommand{\HH}{\mathrm{H}}
\newcommand{\SL}{\mathrm{SL}}

\def \d {{\mathbf d}}
\def \t {{\mathbf t}}

\def \ZZ {{\mathbb Z}}
\def \C {\mathbb{C}}
\def \NN {\mathbb{N}}

\def \mGrass {S_n^m}
\def \jGrass {S_n^j}

\DeclareMathOperator{\D}{\textbf{D}}
\DeclareMathOperator{\Inv}{\mathrm Inv}

\edef\savecatcodeat{\the\catcode`@}
\catcode`\@=11

\def\tb@ifSpecChars#1#2{#1}
\def\tb@ifNoSpecChars#1#2{#2}

\def\tableau{%
  \bgroup
  \@ifstar{\let\Tif\tb@ifNoSpecChars\tb@tableauB}
          {\let\Tif\tb@ifSpecChars\tb@tableauB}}

\def\tb@tableauB{
  \@ifnextchar[{\tb@tableauC}{\tb@tableauC[]}}

\def\tb@tableauC[#1]{\hbox\bgroup%
    \let\\=\cr
    \def\bl{\global\let\tbcellF\tb@cellNF}%
    \def\tf{\global\let\tbcellF\tb@cellH}
    \dimen2=\ht\strutbox \advance\dimen2 by\dp\strutbox%
    \ifx\baselinestretch\undefined\relax%
    \else%
       \dimen0=100sp \dimen0=\baselinestretch\dimen0%
       \dimen2=100\dimen2 \divide\dimen2 by\dimen0%
    \fi%
    \let\tpos\tb@vcenter
    \tb@initYoung
    \tb@options#1\eoo
    \let\arrow\tb@arrow%
    \dimen0=\Tscale\dimen2%
    \dimen1=\dimen0 \advance\dimen1 by \tb@fframe%
    \lineskip=0pt\baselineskip=0pt
    \def\tb@nothing{}%
    \def\endcellno{$\rss\egroup\bss\egroup}
    \def\endcell{\endcellno\kern-\dimen0}
    \def\begincell{\vbox to\dimen0\bgroup\vss\hbox to\dimen0\bgroup\hss$}%
    \let\overlay\tb@overlay%
    \let\fl\tb@fl%
    \let\lss\hss\let\rss\hss\let\tss\vss\let\bss\vss
    \def\mkcell##1{
        \let\tbcellF\tb@cellD
        \def\tb@cellarg{##1}
        \ifx\tb@cellarg\tb@nothing\let\tb@cellarg\tb@cellE\fi%
        \begincell\tb@cellarg\endcellno
        \tbcellF}
    \let\savecellF\tbcellF
     \Tif{\catcode`,=4\catcode`|=\active}{}\tb@tableauD}%

\let\tb@savetableauD\tableauD
{
    \catcode`|=\active \catcode`*=\active \catcode`~=\active%
    \catcode`@=\active
\gdef\tableauD#1{%
  \Tif{
    \mathcode`|="8000 \mathcode`*="8000%
    \mathcode`~="8000 \mathcode`@="8000%
    \def@{\bullet}%
    \let|\cr
    \let*\tf
    \let~\sk
  }{}%
  \tpos{\tabskip=0pt\halign{&\mkcell{##}\cr#1\crcr}}%
  \global\let\tbcellF\savecellF
  \egroup
  \egroup}
}
\let\tb@tableauD\tableauD
\let\tableauD\tb@savetableauD
\let\tb@savetableauD\undefined

\def\tb@options#1{\ifx#1\eoo\relax\else\tb@option#1\expandafter\tb@options\fi}

\def\tb@option#1{%
  \if#1t\let\tpos\tb@vtop\fi
  \if#1c\let\tpos\tb@vcenter\fi
  \if#1b\let\tpos\vbox\fi
  \if#1F\tb@initFerrers\fi
  \if#1Y\tb@initYoung\fi
  \if#1s\tb@initSmall\fi
  \if#1m\tb@initMedium\fi
  \if#1l\tb@initLarge\fi
  \if#1p\tb@initPartition\fi
  \if#1a\tb@initArrow\fi
}

\def\tb@vcenter#1{\ifmmode\vcenter{#1}\else$\vcenter{#1}$\fi}

\def\tb@vtop#1{\hbox{\raise\ht\strutbox\hbox{\lower\dimen0\vtop{#1}}}}

\def\tb@initPartition{\def\Tscale{.3}}
\def\tb@initSmall{\def\Tscale{1}}
\def\tb@initMedium{\def\Tscale{2}}
\def\tb@initLarge{\def\Tscale{3}}

\def\tb@initArrow{\dimen2=1.25em}

\def\tb@initYoung{%
  \def\tb@cellE{}
  \let\tb@cellD\tb@cellN
  \def\sk{\global\let\tbcellF\tb@cellNF}}
\def\tb@initFerrers{%
  \def\tb@cellE{\bullet}
  \let\tb@cellD\tb@cellNF
  \def\sk{\bullet}}

\tb@initMedium

\def\tb@sframe#1{%
  \vbox to0pt{
    \vss
    \hbox to0pt{%
      \hss
      \vbox to\dimen1{
        \hrule depth #1 height0pt
        \vss
        \hbox to\dimen1{
          \vrule width #1 height\dimen1
          \hss
          \vrule width #1
          }%
        \vss
        \hrule height #1 depth 0in
        }%
      \kern-\tb@hframe
      }%
    \kern-\tb@hframe}}

\def\tb@hframe{.2pt}\def\tb@fframe{.4pt}\def\tb@bframe{2pt}
\def\tb@cellH{\tb@sframe{\tb@bframe}}       
\def\tb@cellNF{}                            
\def\tb@cellN{\tb@sframe{\tb@fframe}}       
\let\tbcellF\tb@cellN                       

\def\tb@rpad{1pt}
\def\tb@lpad{1pt}
\def\tb@tpad{1.8pt}
\def\tb@bpad{1.8pt}

\def\tb@overlay{\endcell\@ifnextchar[{\tb@overlaya}{\begincell}}
\def\tb@overlaya[#1]{\vbox to\dimen0\bgroup%
  \tb@overlayoptions#1\eoo%
  \tss\hbox to\dimen0\bgroup\lss$}
\def\tb@overlayoptions#1{\ifx#1\eoo\relax\else\tb@overlayoption#1\expandafter\tb@overlayoptions\fi}

\def\tb@overlayoption#1{
  \if#1t\def\tss{\vskip\tb@tpad}\let\bss\vss\fi
  \if#1c\let\tss\vss\let\bss\vss\fi
  \if#1b\def\bss{\vskip\tb@bpad}\let\tss\vss\fi
  \if#1l\def\lss{\hskip\tb@lpad}\let\rss\hss\fi
  \if#1m\let\lss\hss\let\rss\hss\fi
  \if#1r\def\rss{\hskip\tb@rpad}\let\lss\hss\fi
}

\def\tb@fl{\endcell\begincell\vrule depth 0pt width \dimen0 height \dimen0 \endcell\begincell}

\@ifundefined{diagram}{}{
\def\tb@arrowpad{.5}

\newoptcommand{\tb@arrow}{\@ne}[2]{%
  \endcell
   \begingroup%
   \let\dg@getnodesize\tb@getnodesize
   \dg@USERSIZE=#1\relax%
   \ifnum\dg@USERSIZE<\@ne \dg@USERSIZE=\@ne \fi%
   \dg@parse{#2}%
   \dg@label{\tb@draw{#1}{#2}}}

\def\tb@getnodesize#1#2#3#4#5{\dimen3=\tb@arrowpad\dimen2 #4=\dimen3 #5=\dimen3\relax}
\def\tb@getnodesize#1#2#3#4#5{\ifnum#2=0\ifnum#3=0\tb@getnodesizetail{#4}{#5}\else\tb@getnodesizehead{#4}{#5}\fi\else\tb@getnodesizehead{#4}{#5}\fi}
\def\tb@getnodesizetail#1#2{\dimen3=.5\dimen2 #1=\dimen3 #2=\dimen3}
\def\tb@getnodesizehead#1#2{\dimen3=.5\dimen2 #1=\dimen3 #2=\dimen3}

\def\tb@draw#1#2#3#4{%
        \dg@X=0\dg@Y=0\dg@XGRID=1\dg@YGRID=1\unitlength=.001\dimen0%
        \dg@LBLOFF=\dgLABELOFFSET \divide\dg@LBLOFF\unitlength%
        \dg@drawcalc
        \begincell
        \let\lams@arrow\tb@lams@arrow
        \begin{picture}(0,0)\begingroup\dg@draw{#1}{#2}{#3}{#4}\end{picture}%
        \endcell
        \endgroup
        \begincell}
}

\def\tb@lams@arrow#1#2{%
 \lams@firstx\z@\lams@firsty\z@
 \lams@lastx#1\relax\lams@lasty#2\relax
 \lams@center\z@
 \N@false\E@false\H@false\V@false
 \ifdim\lams@lastx>\z@\E@true\fi
 \ifdim\lams@lastx=\z@\V@true\fi
 \ifdim\lams@lasty>\z@\N@true\fi
 \ifdim\lams@lasty=\z@\H@true\fi
 \NESW@false
 \ifN@\ifE@\NESW@true\fi\else\ifE@\else\NESW@true\fi\fi
 \ifH@\else\ifV@\else
  \lams@slope
  \ifnum\lams@tani>\lams@tanii
   \lams@ht\ten@\p@\lams@wd\ten@\p@
   \multiply\lams@wd\lams@tanii\divide\lams@wd\lams@tani
  \else
   \lams@wd\ten@\p@\lams@ht\ten@\p@
   \divide\lams@ht\lams@tanii\multiply\lams@ht\lams@tani
  \fi
 \fi\fi
 \ifH@  \lams@harrow
 \else\ifV@ \lams@varrow
 \else \lams@darrow
 \fi\fi
}

\catcode`\@=\savecatcodeat
\let\savecatcodeat\undefined

\definecolor{amethyst}{rgb}{0.6, 0.4, 0.8}
\definecolor{kellygreen}{rgb}{0.3, 0.73, 0.09}
\definecolor{americanrose}{rgb}{1.0, 0.01, 0.24}

\usepackage[colorinlistoftodos]{todonotes}

\newcommand{\red}[1]{\textcolor{red}{#1}}
\newcommand{\blue}[1]{\textcolor{blue}{#1}}

\newcommand{\FlT}{\Gamma_{\operatorname{T}}}

\newcommand{\PC}{\Psi_{\operatorname{PC}}}

\newcommand{\SD}{\Gamma_{\operatorname{SD}}}

\newcommand{\LJGr}{\Phi_{\operatorname{Gr}}}

\newcommand{\LJFl}{\Phi_{\operatorname{Fl}}}

\begin{document}

\title[An affine approach to Peterson comparison]{An affine approach to Peterson comparison}

\author{Linda Chen}
\address{Department of Mathematics and Statistics, Swarthmore College, Swarthmore, PA 19081}
\email{lchen@swarthmore.edu}
\author{Elizabeth Mili\'cevi\'c}
\address{Department of Mathematics and Statistics, Haverford Colllege, Haverford, PA 19041}
\email{emilicevic@haverford.edu}
\author{Jennifer Morse}
\address{Department of Mathematics, University of Virginia, Charlottesville, VA 22904}
\email{morsej@virginia.edu}


\thanks{LC was partially supported by Simons Collaboration Grant 524354. EM was partially supported by NSF Grant DMS-1600982 and Simons Collaboration Grant 318716. JM was partially supported by NSF Grants DMS-1855804 and DMS-1833333.}

\begin{abstract}
The Peterson comparison formula proved by Woodward relates the three-pointed Gromov-Witten invariants for the quantum cohomology of partial flag varieties to those for the complete flag.  Another such comparison can be obtained by composing a combinatorial version of the Peterson isomorphism with a result of Lapointe and Morse relating quantum Littlewood-Richardson coefficients for the Grassmannian to $k$-Schur analogs in the homology of the affine Grassmannian obtained by adding rim hooks.  We show that these comparisons on quantum cohomology are equivalent, up to Postnikov's strange duality isomorphism.
\end{abstract}

\maketitle


\section{Introduction}\label{sec:intro}
The study of quantum cohomology emerged from physics, and connections to enumerative geometry drew attention to the mathematical ideas being employed by the superstring theorists \cite{Witten}.  The axiomatic and functorial development of Gromov-Witten theory \cite{KontManin,GivKim} then paved the way for algebraic geometers to reformulate these problems from the moduli space perspective \cite{FP}. The resulting combinatorial pursuit of quantum Schubert calculus, which is the topic of the present paper, aims to explicitly describe the product structure in the (small) quantum cohomology ring $\QH^*(G/P)$, where $G$ is a connected complex reductive group, and $P$ is a parabolic subgroup containing a fixed Borel $B$.

\subsection{Comparing quantum cohomology rings}

Unlike classical cohomology, quantum cohomology is not functorial.  More precisely, the natural projection $G/B \rightarrow G/P$ does not give rise to a map on quantum cohomology $\QH^*(G/P) \rightarrow \QH^*(G/B)$, as it does in the classical case.  Nevertheless, scattered throughout the literature are several methods for relating the structure constants  for these two rings, in an effort to partially restore functoriality.  The structure constants are given by three-pointed Gromov-Witten invariants, or quantum Littlewood-Richardson coefficients. One such result is the Peterson comparison formula \cite{Pet}, which was proved by Woodward \cite{w} using the geometry of principal bundles.

Another celebrated result of Peterson \cite{Pet}, proved by Lam and Shimozono \cite{ls}, equates the studies of the quantum cohomology of the homogeneous space $G/P$ and the homology of the affine Grassmannian $\AfGr_G$. In \cite{LamJAMS}, Lam showed that the Schubert basis for the ring $\HH_*(\AfGr_{\SL_n})$ is represented by the $k$-Schur functions of  \cite{LMktab}. Combining this quantum-to-affine correspondence of Peterson with results of Lapointe and Morse from \cite{lm} then yields a method for using $k$-Littlewood-Richardson coefficients to compare the quantum cohomology rings for complete and partial flag varieties in type $A$.  

The goal of this paper is to explain the precise relationship between these different means for comparing products in $\QH^*(G/P)$ and $\QH^*(G/B)$, in the special case where $G=\SL_n$ and $G/P$ is the Grassmannian.  A critical component of this connection arises from an unexpected symmetry on $\QH^*(G/P)$ discovered by Postnikov for the Grassmannian \cite{p1}, and generalized to other (co)miniscule homogeneous spaces by Chaput, Manivel, and Perrin \cite{cmp07}. The Main Theorem provides an affine approach to the Peterson comparison formula via this strange duality isomorphism.

\subsection{Statement of the Main Theorem}

The Schubert basis for the quantum cohomology $\QH^*(\Gr_{m,n})$ is indexed by partition shapes contained in a fixed rectangle, whereas the Schubert basis for $\QH^*(\Fl_n)$ is indexed by permutations.  The primary goal of quantum Schubert calculus is to provide combinatorial formulas for the quantum Littlewood-Richardson coefficients describing the product of two Schubert classes; see Sections \ref{sec:GrGWs} and \ref{sec:FlGWs} for more details.

The Peterson comparison formula \cite{w} relates the quantum Littlewood-Richardson coefficients for $G/P$ to those for $G/B$, where $G$ is any connected, simply connected, semisimple complex reductive group.  We specialize this comparison in Section \ref{sec:PC} as $\PC$, which expresses every quantum Littlewood-Richardson coefficient for $\QH^*(\Gr_{m,n})$ in terms of certain ones for $\QH^*(\Fl_n)$.  This statement is formalized in Theorem \ref{thm:comparison}, and several related ingredients appeared in the work of Leung and Li on graded filtrations of subspaces in $\QH^*(\Fl_n)$; see \cite{LLiJDG, LLiIMRN}.
To keep our paper self-contained, we provide an independent proof of Theorem \ref{thm:comparison}  in Section \ref{sec:comparison}.

For $G=\SL_n$, the affine Grassmannian is $\AfGr_n = G(\C((t))/G(\C[[t]])$.  
The Schubert basis for the homology $\HH_*(\AfGr_n)$ is represented by $k$-Schur functions, which are indexed by $k$-bounded partitions. A primary goal of affine Schubert calculus is to provide combinatorial formulas for the $k$-Littlewood-Richardson coefficients, which determine the product on $k$-Schur functions. The Peterson isomorphism implies that all quantum Littlewood-Richardson coefficients are encoded as $k$-Littlewood-Richardson coefficients, and so this affine Schubert problem strictly contains the quantum one; see Section \ref{sec:affLRs} for more details.

In \cite{lm}, Lapointe and Morse show that the quantum Littlewood-Richardson coefficients for $\QH^*(\Gr_{m,n})$ are equal to those $k$-Littlewood Richardson coefficients obtained by adding rim hooks to the corresponding $k$-bounded partition. We review this result as $\LJGr$ in Theorem \ref{thm:LucJenGr} from Section \ref{sec:LucJenGr}. Directly from the Peterson isomorphism, every product in $\QH^*(\Fl_n)$ can be expressed in many ways using the $k$-Littlewood-Richardson coefficients, one of which is reviewed in Theorem \ref{thm:lstodarephrase}, rephrased from \cite{lstoda} following the treatment in \cite{BMPSadvances}. Importantly for our purposes, this correspondence from Theorem \ref{thm:lstodarephrase} is invertible on the image of $\LJGr$. We make this claim precise in Section \ref{sec:AffComp}, the primary goal of which is to prove Theorem \ref{thm:LJFlcomb} formalizing $\LJFl$.

The Main Theorem describes the precise connection between the Peterson comparison formula $\PC$ and the composition of the quantum-to-affine correspondences $\LJFl \circ \LJGr$, both of which provide a means for directly relating the rings $\QH^*(\Gr_{m,n})$ and $\QH^*(\Fl_n)$ in the absence of functoriality.  Though not identical, these comparisons differ by exactly two duality isomorphisms, one of which is the standard flag transpose $\FlT$ reviewed in Section \ref{sec:FlT}.  The more subtle discrepancy arises from the strange duality isomorphism $\SD$ on $\QH^*(\Gr_{m,n})$ from \cite{p1}, which we review in Section \ref{sec:SD} in the special case recorded as Theorem \ref{thm:SDPost}.  These various comparisons $\QH^*(\Gr_{m,n}) \rightarrow \QH^*(\Fl_n)$ are related as follows.

\begin{mainthm}
For any $m,n,r\in \NN$ such that $m+r=n$, the following diagram on Littlewood-Richardson coefficients commutes: 
\begin{center}
\begin{tikzcd}[row sep=tiny]
& \QH^*(\Gr_{m,n}) \arrow[r,"\PC"] & \QH^*(\Fl_n) \arrow[dd, leftrightarrow, "\FlT"]  \\
\QH^*(\Gr_{m,n}) \arrow[ur, leftrightarrow, "\SD"] \arrow[dr,"\LJGr"'] & &  \\
& \HH_*(\AfGr_n) \arrow[r, rightarrow, "\LJFl"] & \QH^*(\Fl_n)
\end{tikzcd}
\end{center}
e.g.\ for any partitions $\lambda, \mu, \nu \subseteq (r)^m$ and any $d \in \ZZ_{\geq0}$ such that $|\lambda| + |\mu| = |\nu| + nd$, given the quantum Littlewood-Richardson coefficient $c_{\lambda, \mu}^{\nu,d}$ in $\QH^*(\Gr_{m,n})$, we have
\[ \FlT \circ \PC \circ \SD \left( c_{\lambda, \mu}^{\nu, d}\right) = \LJFl \circ \LJGr  \left( c_{\lambda, \mu}^{\nu, d}\right). \]
\end{mainthm}

The Main Theorem is made precise in Section \ref{sec:background}, which provides self-contained statements for each of the five correspondences that appear in the diagram above, illustrated by a common running example. The two relationships original to this paper are $\PC$ which is the subject of Section \ref{sec:comparison}, and $\LJFl$ which is formalized in Section \ref{sec:AffComp}. The proof of the Main Theorem is largely combinatorial in nature and follows in Section \ref{sec:proof}.

\subsection{Discussion of related and future work}

We conclude by highlighting several similar results which equate certain quantum and/or affine Littlewood-Richardson coefficients, discussing their relationship to our Main Theorem, and mentioning some related open problems.

The quantum cohomology ring $\QH^*(\Gr_{m,n})$ has been well-studied, 
with Pieri and Giambelli formulas established using geometric techniques of Bertram  \cite{Bertram}, 
an elementary linear-algebraic approach of Buch \cite{Buch}, and various other formulations which followed.
However, there is still much work to be done towards understanding the combinatorics of more general
 Littlewood-Richardson coefficients; even the classical coefficients for $\HH^*(\Fl_n)$ remain elusive.
It would be interesting to attack this problem with the connection to $k$-Littlewood-Richardson coefficients via the correspondence $\LJFl$.
A natural starting point would be to capitalize on the well-developed case of $\QH^*(\Gr_{m,n})$. Formulas for these quantum Littlewood-Richardson coefficients could be
traced through $\LJGr$, and the outcome compared to combinatorics supporting 
$k$-Schur functions.  
The correspondence $\LJGr$ identifies these quantum Littlewood-Richardson coefficients with the $k$-Littlewood-Richardson 
coefficients arising in the $k$-Schur expansion of a product of ordinary Schur functions.  
From there, similar comparisons with known special cases for $\QH^*(\Fl_{n})$ to those obtained from $\LJFl$ 
and its inverse could be made.

In \cite{BCFF},  Bertram, Ciocan-Fontanine, and Fulton give a rim hook algorithm for computing  quantum Littlewood-Richardson coefficients for $\QH^*(\Gr_{m,n})$ in terms of (signed) classical Littlewood-Richardson coefficients. The rim hooks occurring in the correspondence $\LJGr$ are a special case of those in \cite{BCFF}, though the comparison $\LJGr$ is a quantum-to-affine one, rather than quantum-to-classical.  That is, the correspondence $\LJGr$ expresses the quantum Littlewood-Richardson coefficients for $\QH^*(\Gr_{m,n})$  in terms of certain (positive) $k$-Littlewood-Richardson coefficients.  As shown in Theorem 1.1 of \cite{CM}, the rim hooks in $\LJGr$ correspond precisely to translation elements which are localized in the parabolic Peterson isomorphism comparing $\HH_*(\AfGr_n) $ to $\QH^*(\Gr_{m,n})$, suggesting an approach to generalizing $\LJGr$ to other partial flag varieties.

In \cite{BKT2step},  Buch, Kresch, and Tamvakis identify  the quantum Littlewood-Richardson coefficient  $c_{\lambda, \mu}^{\nu,d}$ for $\QH^*(\Gr_{m,n})$ with a classical Littlewood-Richardson coefficient for the two-step flag variety whose Schubert classes are indexed by permutations with descents in positions $m\pm d$.  In \cite[Section 2]{LLiIMRN}, Leung and Li show that this ``quantum-to-classical" formula of  \cite{BKT2step} can be recovered from the Peterson comparison formula. Although the third permutation indexing the quantum Littlewood-Richardson coefficient in $\QH^*(\Fl_n)$ under the image of $\PC$  as formulated in Theorem \ref{thm:comparison} also has descents in positions $m \pm d$, this does not give a direct relation between $\PC$ and the comparison to two-step flag varieties in \cite{BKT2step} and \cite{LLiIMRN}.  In particular, quantum Littlewood-Richardson coefficients for $\QH^*(\Gr_{m,n})$ of nonzero degree always compare via $\PC$ to quantum Littlewood-Richardson coefficients in $\QH^*(\Fl_n)$ with nonzero degree,  and conversely,  the Peterson comparison formula applied to the classical two-step Littlewood-Richardson coefficients arising in \cite{BKT2step} gives degree zero Littlewood-Richardson coefficients in $\QH^*(\Fl_n)$.

Lam and Shimozono observe in Proposition 11.10 of \cite{ls} that at $q=1$, the correspondence $\LJGr$ of Lapointe and Morse is the composition of the parabolic Peterson isomorphism and Postnikov's strange duality $\SD$.  Another perspective on the Main Theorem is that it completes this story for arbitrary $q$, using the fact that the relation $\LJFl$ is invertible. It would be natural to explore analogs of these relationships in other types, where one would expect the strange duality of Chaput, Manivel, and Perrin \cite{cmp07} to be the critical link.  In type $A$, the Main Theorem also suggests a roadmap for possible generalizations of $\SD$ to partial flag varieties beyond the Grassmannian, for which there is no known analog of strange duality.  An alternative path to exploring generalizations of the Main Theorem would be to work on the level of Schubert classes following Cookmeyer and Mili\'cevi\'c, using the parabolic Peterson isomorphism directly as in Theorem 1.2 of \cite{CM}.

\subsection*{Acknowledgements}

 EM gratefully acknowledges the support of the Max-Planck-Institut f\"ur Mathematik, which hosted two long-term sabbatical visits in 2016 and 2020, during which significant portions of this this work were completed. The authors wish to thank the anonymous referee for pointing out several additional references.


\section{Background on Quantum and Affine Schubert calculus}
\label{sec:background}

The purpose of this section is to formally state each of the five different equalities of quantum and/or affine Littlewood Richardson coefficients appearing in the Main Theorem, and to develop the combinatorial background for the corresponding quantum and affine Schubert calculus; see references such as \cite{Fulton,FGP,Buch,kSchur} for more details. Each correspondence is illustrated by a common running example, and the relationships among these comparisons as stated in the Main Theorem is then demonstrated at the end of the section. Throughout the paper, we fix integers $m,n,r \in \NN$ such that $m+r=n$ and define $k=n-1$.

\subsection{Quantum Littlewood-Richardson coefficients for the Grassmannian}\label{sec:GrGWs}

The Grassmannian of $m$-dimensional subspaces of $\C^n$ shall be denoted by $\Gr_{m,n}$. The cohomology ring $\HH^*(\Gr_{m,n})$ has a basis of Schubert classes, indexed by \emph{partitions}  $\lambda=(r\geq \lambda_1\geq\dots\geq\lambda_m\geq 0)$.  A partition $\lambda$ is typically represented in French notation as a {\it Ferrers shape} with $\lambda_i$ boxes or cells in the $i^{\text{th}}$ row, 
where the indexing is inherited from the embedding of $\lambda$ in the $\NN \times \NN$ plane;  
row 1 is the lowest row and row $m$ is the highest.
The  number of rows of the partition is denoted by $l(\lambda)=m$. We use this correspondence between partitions and Ferrers shapes without comment.  Any two partitions can be added coordinate-wise, inserting trailing zeros as necessary in order that the number of parts is equal.

The Schubert basis for $\HH^*(\Gr_{m,n})$ is then indexed by partitions which lie inside the 
rectangle $R_r:=(r^{m})$ having $r$ columns and $m$ rows;
we use $r^m$ to represent $m$ copies of a row of size $r$ throughout.
For $\lambda\subseteq R_r$, the corresponding Schubert class $\sigma_\lambda\in \HH^{2| \lambda|}(\Gr_{m,n})$, where $| \lambda|$ is the total number of boxes in $\lambda$. When $\lambda\subseteq R_r$, its {\it complement} is the partition
\begin{equation}\label{eq:complement}
\lambda^{\vee}=(r-\lambda_{m},\ldots,r-\lambda_{1})\,.
\end{equation}
If there is any ambiguity about the underlying rectangle containing $\lambda \subseteq R_r$, we use 
the notation $\lambda^{\vee_{r}}$.
The Schubert classes $\sigma_{\lambda^\vee}$ form a Poincar\'e dual basis, meaning
$\sigma_\mu \cdot \sigma_{\lambda^\vee}  = \delta_{\mu,\lambda}$
in $\HH^*(\Gr_{m,n})$.

The \emph{quantum cohomology ring}  $\QH^*(\Gr_{m,n})$ is a commutative and associative graded algebra over $\ZZ[q]$, where $q$ is a parameter of degree $n$.  As a $\ZZ[q]$-module, the quantum cohomology $\QH^*(\Gr_{m,n}) := \ZZ[q]\otimes \HH^*(\Gr_{m,n})$, and thus also has a basis of Schubert classes indexed by partitions $\lambda \subseteq R_r$, which we again denote by $\sigma_\lambda$ since the context should always be clear; i.e.\ 
\[
  \QH^*(\Gr_{m,n}) = \bigoplus_{\lambda \subseteq R_r} \ZZ[q]\sigma_\lambda.
\]
The  quantum product $*$ is a deformation of the classical product $\cdot$ in $\HH^*(\Gr_{m,n})$.  Given partitions $\lambda,\mu\subseteq R_r$,  the \emph{quantum Littlewood-Richardson coefficients} $c_{\lambda,\mu}^{\nu,d}$ for the Grassmannian are defined by
\begin{equation*}
  \sigma_\lambda * \sigma_\mu = \sum_{\nu,d}c_{\lambda,\mu}^{\nu,d}\, q^d\, \sigma_\nu,
\end{equation*}
where the sum ranges over $\nu \subseteq R_r$ and $d\in \ZZ_{\geq 0}$. 
By degree considerations, $c_{\lambda,\mu}^{\nu,d}=0$ unless $|\lambda|+|\mu|=|\nu|+nd$.
Note that the quantum Littlewood-Richardson coefficient $c_{\lambda,\mu}^{\nu,d}$ equals the three-point Gromov-Witten invariant $\langle \sigma_\lambda,\sigma_\mu,\sigma_{\nu^\vee}\rangle_d$.

\subsubsection{Strange duality for the Grassmannian}\label{sec:SD}

There is a \emph{strange duality} isomorphism on the quantum cohomology ring, first proved by Postnikov for Grassmannians \cite{p1}, and then generalized by Chaput, Manivel, and Perrin to all (co)miniscule homogeous spaces \cite{cmp07}. This duality inverts the quantum parameter $q$, and thus curiously places rational curves on $\Gr_{m,n}$ of high degree in bijection with those of low degree.  We review a version of this strange duality isomorphism $\SD$ from \cite{p1} in Theorem \ref{thm:SDPost} below.

To state this result, one required statistic on partitions is the length of the the main diagonal $\diag_0(\lambda)$, which denotes the number of boxes of $\lambda$ with equal row and column index. 
In addition, each partition $\lambda \subseteq R_{r}$ can be uniquely identified with a {\it bit string} $b_\lambda \in \{0,1\}^n$ having $m$ zeros and $r$ ones.  To construct $b_\lambda$, trace the boundary of the shape $\lambda$, starting from the upper left corner of $R_{r}$, recording each vertical step as 0 and each horizontal step as 1; each 0 and 1 in the resulting string is referred to as a {\it bit}.  We also use this bijective correspondence between shapes and bit strings freely without comment.

For any integer $1 \leq a \leq n$, define the {\it cycling map} $\cyc^a$ to act on the bit string $b_\lambda$ by cycling the first $a$ bits of $b_\lambda$ to the end of the string. Note that $\cyc^n$ is the identity map, and that the inverse of the map $\cyc^a$ is given by $\cyc^{-a} = \cyc^{n-a}$.

\begin{theorem}\cite[Corollary 6.8]{p1}\label{thm:SDPost}
\label{thm:strangeduality}
For any partitions $\lambda,\mu,\nu \subseteq R_r$ and any integer $d \in \ZZ_{\geq0}$, we have 
\begin{equation*}
c_{\lambda,\mu}^{\nu,d} \xlongequal{\SD} c_{\lambda^\vee, \mu^\vee}^{\cyc^{r}(\nu)^\vee,t},
\end{equation*}
where $t = \diag_0(\nu^\vee) - d$.
\end{theorem}

\begin{proof} For any partition $\eta \subseteq R_{r}$, the bit string $b_\eta=b_1b_2\cdots b_n$ relates to its complement by reversing the bits $b_{\eta^\vee}=b_n\cdots b_2b_1$.  Therefore, the bit strings of $\cyc^r(\nu)^\vee$ and $\cyc^{-r}(\nu^\vee)=\cyc^{m}(\nu^\vee)$ coincide, and the result now follows immediately from the equality  $c_{\lambda,\mu}^{\nu,d} = c_{\lambda^\vee, \mu^\vee}^{\cyc^{m}(\nu^\vee),t}$ in Corollary 6.8 of \cite{p1}.
\end{proof}

We now consider an example which we will use to illustrate each of the theorems in this section, as well as their relationship as stated in the Main Theorem.

\begin{example}\label{ex:SDfav}
Let $n=5$ and $r=2$, and consider the partitions $\lambda=(2,2,1)$ and $\mu = (1,1,0)$. Using the quantum Pieri formula \cite{Bertram}, we can calculate the following product in $\QH^*(\Gr_{3,5})$:
\[ \sigma_{\tableau[pcY]{|,|,}} * \sigma_{\tableau[pcY]{||}} = q\sigma_{\tableau[pcY]{||}} + q\sigma_{\tableau[pcY]{,}}. \]
For the partition $\nu = (2,0,0) \subset R_2$ indexing the second summand and $d=1$, we see that $c_{\lambda, \mu}^{\nu,d} = 1$.
To apply $\SD$, we compute that $b_\nu = 00110$ so that $\cyc^2(\nu) = 11000 \leftrightarrow (2,2,2)$.  Therefore $\cyc^2(\nu)^\vee = (0,0,0)$ is the empty shape.  In addition, $\nu^\vee = (2,2,0)$ so that $\diag_0(\nu^\vee) = 2$ and $d' = 2-1=1$.  Finally, $\lambda^\vee = (1,0,0)$ and $\mu^\vee = (2,1,1)$.  Theorem \ref{thm:SDPost} then says that 
\[ c_{\,\tableau[pcY]{|,|,}\, , \, \tableau[pcY]{||}}^{\, \tableau[pcY]{,}\, , \,1} \xlongequal{\SD} c_{\tableau[pcY]{|}\ , \ \tableau[pcY]{||,}}^{\emptyset\, , \, 1}. \]
\end{example}

\subsection{Quantum Littlewood-Richardson coefficients for the flag variety}\label{sec:FlGWs}

The complete flag variety for $\C^n$ shall be denoted by $\Fl_n$.  The cohomology ring $\HH^*(\Fl_n)$ has a basis of Schubert classes indexed by elements of the symmetric group $S_n$, which are permutations on the set $[n]:=\{1,\dots,n\}$.  The {\it one-line} or \emph{window} notation for a permutation records the action of $w \in S_n$ on the elements of $[n]$ as $w = [w_1 \cdots w_n],$ where $w_i=w(i)$. Given $w \in S_n$, the corresponding Schubert class $\sigma_w\in \HH^{2\ell(w)}(\Fl_n)$, where the \emph{length} of $w\in S_n$ is the number of inversions 
\begin{equation*}
  \ell(w) = \#\{ i<j \mid w(i)>w(j) \}.
\end{equation*}
There is a unique element of greatest length in $S_n$;  this longest element is denoted by $w_0$ and  is defined as a permutation by $w_0(i) = n+1-i$ for all $i \in [n]$.
The Schubert classes $\sigma_{w_0 w}$ form a Poincar\'{e} dual basis, meaning $\sigma_v\cdot \sigma_{w_0 w} =\delta_{v,w}$ in $\HH^*(\Fl_n)$. 

The \emph{quantum cohomology ring} $\QH^*(\Fl_n)$ is a commutative and associative graded algebra over $\ZZ[{\mathbf q}]:=\ZZ[q_1,\ldots,q_{k}]$, where each $q_i$ is a parameter of degree $2$.  As a $\ZZ[{\mathbf q}]$-module, the quantum cohomology $\QH^* (\Fl_n):= \ZZ[{\mathbf q}]\otimes \HH^*(\Fl_n)$, and thus also has a basis of Schubert classes, which we again denote by $\sigma_w$.  Given permutations $u,v\in S_n$,  the (generalized) \emph{quantum Littlewood-Richardson coefficients} $c_{u,v}^{w,d}$ for the flag variety are defined by
\begin{equation*}
  \sigma_u * \sigma_v = \sum_{w,\d} c_{u,v}^{w,\d}\, {\mathbf q}^\d\, \sigma_w,
\end{equation*}
where the sum ranges over $w \in S_n$ and $\d = (d_1, \dots, d_{k}) \in \ZZ^{k}_{\geq 0}$, and we denote by ${\mathbf q}^{\d} = q_1^{d_1} \cdots q_{k}^{d_{k}}$. 
By degree considerations, $c_{u,v}^{w,\d}=0$ unless $\ell(u)+\ell(v)=\ell(w)+2|\d|$, where $|\d| = \sum_{i=1}^{k} d_i$.
Note that the quantum Littlewood-Richardson coefficient $c_{u,v}^{w,\d}$ equals the three-point Gromov-Witten invariant $\langle  \sigma_u,\sigma_v,\sigma_{w_0w}\rangle_\d$.

\subsubsection{The Peterson comparison formula}\label{sec:PC}

The Peterson comparison formula stated by Peterson \cite{Pet} and proved by Woodward \cite{w} relates the quantum Littlewood-Richardson coefficients for the homogeneous space $G/P$ to those of $G/B$, where $G$ is any connected, simply connected, semisimple complex reductive group. In Theorem \ref{thm:PetComp}, we specialize the Peterson comparison formula for $G=\SL_n$ to equate the quantum Littlewood-Richardson coefficients for the Grassmannian to certain ones for the complete flag variety via $\PC$. 

Every partition $\lambda\subseteq R_r$ can be uniquely identified with a permutation $w_\lambda\in S_n$ defined by  
\begin{equation}\label{eq:lampermdef}
w_\lambda(i)=\lambda_{m-i+1}+i
\end{equation}
 for $1\leq i\leq m$, and then ordering the remaining values $w_\lambda(m+1)<\cdots<w_\lambda(n)$.  
A permutation $w\in S_n$ has a {\it descent at $i$}  if $w(i)>w(i+1)$, 
and $w$ is called a \emph{Grassmann permutation} if it has at most one descent. 
The map $\lambda\mapsto w_\lambda$ gives a bijection between partitions in $R_r$ 
and the set of Grassmann permutations with a descent at $m$, which we denote by $\mGrass$
(and the empty partition maps to the identity permutation).
The inverse map is defined by $w \mapsto \lambda_w$, where
\begin{equation}\label{eq:permlamdef}
\lambda_w= (w(m)-m,\ldots,w(1)-1)\subseteq R_r.
\end{equation}
The subgroup $S_m \times S_r$ is the Weyl group $W_P$ for the maximal parabolic subgroup $P$ such that $\Gr_{m,n} \cong \SL_n(\C)/P$, and the longest element of this subgroup is 
\begin{equation}\label{eq:w0perm}
w_0^{P} = [ m \cdots  1 \mid n \cdots m+1].
\end{equation}
Note that Poincar\'e duality in $\QH^*(\Gr_{m,n})$ can also be realized via the relation 
\begin{equation}\label{eq:dual-perm}
w_{\lambda^\vee}=w_0w_\lambda w_0^P.
\end{equation}

We specialize the Peterson comparison formula from \cite{w} to compare the quantum Littlewood-Richardson coefficients for $\QH^*(\Gr_{m,n})$ and $\QH^*(\Fl_n)$ in Theorem \ref{thm:PetComp}; the proof of Theorem \ref{thm:PetComp} follows in Section \ref{sec:comparison}.  We remark that the key combinatorial identities in the proof of Theorem \ref{thm:comparison} appeared in the work of Leung and Li, albeit using somewhat different language; see \cite[Lemma 3.6]{LLiJDG} and \cite[Section 2.3]{LLiIMRN}.

\begin{theorem} 
\label{thm:comparison}\label{thm:PetComp}
Let $u,v,w \in \mGrass$, and fix any integer $0 \leq d \leq \min \{r, m\}$.  Then
\begin{equation*}
c_{\lambda_u,\lambda_v}^{\lambda_w,\,d} \xlongequal{\PC} c_{u,v}^{ww_0^Pw_0^{P'_d},{\d}} 
\end{equation*}
where 
\begin{align}
\d & = (0^{m-d}, 1, 2, \dots, d-1, d, d-1, \dots, 2, 1, 0^{r-d}) \in \ZZ_{\geq 0}^{k}, \quad \text{and} \label{eq:dPC} \\
w_0^{P'_d} & = [ m-d \cdots 1 \mid m  \cdots m-d+1 \mid m+d \cdots m+1 \mid n \cdots m+d+1]. \label{eq:w0P'}
 \end{align}
\end{theorem}

\begin{example}\label{ex:PCfav}
To illustrate the statement of the Main Theorem, we continue by applying Theorem \ref{thm:PetComp} to the result of Example \ref{ex:SDfav} illustrating $\SD$.  We thus have $\lambda_u = (1,0,0)$ and $\lambda_v = (2,1,1) \subset R_2$.  Applying \eqref{eq:lampermdef}, we recover the permutations $u = [ 12435]$ and $v = [23514]$, each of which has a single descent in position $n-r = 3$. The empty shape $\lambda_w = (0,0,0)$
produces the identity permutation $w = [12345]$. 

Specializing formula \eqref{eq:w0perm} for the longest element of $S_3 \times S_2$, we have $w_0^P = [ 32154]$, and applying formula \eqref{eq:w0P'} for $w_0^{P'_d}$ with $d=1$ gives us $w_0^{P'_d}=[21345]$. Composing, we obtain $ww_0^Pw_0^{P'_d} = [23154]$. Finally, we have $\d = (0^{3-1}, 1, 0^{2-1}) = (0,0,1,0)$ by Equation \eqref{eq:dPC}.  Therefore, Theorem \ref{thm:PetComp} says that 
\[ c_{\tableau[pcY]{|}\ , \ \tableau[pcY]{||,}}^{\emptyset, 1}  \xlongequal{\PC}  c_{[12435]\, ,\, [23514]}^{[23154]\, ,\, (0,0,1,0)}.\]
\end{example}

\subsubsection{Permutations and the transpose map}\label{sec:FlT}

The natural map $\Fl_n\rightarrow\Fl_n$ which identifies  an $r$-dimensional subspace of $V=\C^n$ with  an $m$-dimensional subspace of $V^*$  induces an isomorphism on $\QH^*(\Fl_n)$, which acts on Schubert classes as conjugation by $w_0$. We denote this induced map by
\begin{align*}
\FlT : \QH^*(\Fl_n) & \rightarrow \QH^*(\Fl_n)
\\
\sigma_w & \mapsto \sigma_{w'}, \nonumber
\end{align*}
where for a permutation $w=[w_1 \cdots w_n] \in S_n$ in one-line notation, the conjugate permutation 
$w'=w_0ww_0$ is given by $w'(n+1-i)=n+1-w(i)$ for all $i \in [n]$.   We then have the following equality of quantum Littlewood-Richardson coefficients
\begin{equation}\label{eq:FlT}
c_{u,v}^{w,\d} \xlongequal{\FlT} c_{u',v'}^{w',\d'},
\end{equation}
where the vector $\d' \in \ZZ_{\geq 0}^{k}$ is defined by $d'_i=d_{n-i}$ for $1\leq i\leq k$.  We refer to an application of $\FlT$ as the \emph{flag transpose}.

The notation $\FlT$ is derived from the fact that the corresponding map on $\QH^*(\Gr_{m,n})$ acts on a Schubert class by $\sigma_\lambda \mapsto \sigma_{\lambda '}$, where $\lambda'$ is the \emph{transpose} obtained by exchanging the rows and columns of $\lambda$. Note that
\begin{equation}\label{eq:conjugate}
 w_{\lambda'}= w_0w_\lambda w_0=w'_\lambda.
\end{equation}

\begin{example}\label{ex:FlTfav}
Consider $u = [12435], v = [23514], w = [23154] \in S_5$ and $\d = (0,0,1,0)$ from Example \ref{ex:PCfav} illustrating $\PC$.  Applying the flag transpose on $\QH^*(\Fl_5)$ and recalling that $\pi'(n+1-i)=n+1-\pi(i)$ for any $\pi \in S_n$ gives us 
\[ c_{[12435]\, ,\, [23514]}^{[23154]\, ,\, (0,0,1,0)}  \xlongequal{\FlT}   c_{[13245]\, ,\, [25134]}^{[21534]\, ,\, (0,1,0,0)}.\]
\end{example}

We now collect some facts  relating the transpose to the bijection between partitions and Grassmann permutations that we will use in Section \ref{sec:proof}.

\begin{lemma} \label{le:comtra}
For $\lambda=(\lambda_1,\dots,\lambda_j)\subseteq R_{n-j}$, write  $\lambda'=(\lambda'_1,\dots,\lambda'_{n-j})\subseteq R_j$ for the transpose partition to $\lambda$, 
and consider the transpose dual partition $(\lambda')^{\vee_j}=(\lambda^{\vee_{n-j}})' \subseteq R_j$. Then 
\begin{align*}
w_\lambda&= [\lambda_j+1\cdots\lambda_{1}+j\mid j+1-\lambda'_1\cdots n-\lambda'_{n-j}],\\
w_{({\lambda^{\vee_{n-j}})'}}&=
[j+1-\lambda'_1\ddots n-\lambda'_{n-j}\mid\lambda_j+1\cdots\lambda_{1}+j].
\end{align*}
\end{lemma}

\begin{proof}
Recall from \eqref{eq:lampermdef} that for $\lambda\subseteq R_{n-j}$, the first $j$ entries of the one-line notation of its corresponding permutation $w_\lambda= [ w_1 \cdots w_{j} \mid w_{j+1}\cdots w_n]\in \jGrass$ are
\begin{equation}\label{eq:wlambda-partial}
w_\lambda =[w_1\cdots w_j\mid\cdots]= [\lambda_j+1\dots\lambda_{1}+j\mid \cdots].\end{equation}
Consider  $(\lambda^{\vee_{n-j}})'$, the dual to the transpose of the partition  $\lambda$. Combining  \eqref{eq:conjugate} with  \eqref{eq:dual-perm} applied to the parabolic subgroup $P_j$ such that $\Gr_{j,n} \cong \SL_n(\C)/P_j$, we have
$$w_{(\lambda^{\vee_{n-j}})'} = w_0 w_{\lambda^{\vee_{n-j}}}w_0 = w_0(w_0w_\lambda w_0^{P_j} )w_0 = w_\lambda  w_0^{P_j}w_0, $$
where $w_0^{P_j} =[j \cdots  1 \mid n \cdots j+1]$ so that $w_0^{P_j}w_0=[j+1\cdots n \mid 1\cdots j]$, and therefore
\begin{equation} \label{eq:wdualtranspose}
w_{(\lambda^{\vee_{n-j}})'} = [w_{j+1}\cdots w_n\mid w_1 \cdots w_{j}].
\end{equation}
The complement to $\lambda'$ in $R_j$ is given by
$({\lambda'})^{\vee_j} = (j-\lambda'_{n-j},\dots,j-\lambda'_1)\subseteq R_j$,
so by  \eqref{eq:lampermdef},  we obtain the first $n-j$ entries of the one-line notation
of its corresponding permutation
\begin{equation} \label{eq:wdualtranspose-partial} w_{({\lambda'})^{\vee_j}} =[ (j-\lambda'_1)+1\cdots (j-\lambda'_{n-j})+(n-j) \mid \cdots]
 =[w_{j+1}\cdots w_n\mid\cdots].\end{equation}
Since $w_{({\lambda'})^{\vee_j}} = w_{({\lambda^{\vee_{n-j}})'}}$, comparing  \eqref{eq:wdualtranspose}  with the values of $w_1,\ldots,w_n$ identified in \eqref{eq:wlambda-partial} and \eqref{eq:wdualtranspose-partial} gives the result.
\end{proof}

\subsection{Affine Littlewood-Richardson coefficients}\label{sec:affLRs}

For $G=\SL_n$, the \emph{affine Grassmannian} is defined as $\AfGr_n = G(\C((t))/G(\C[[t]])$. The homology ring  $\HH_*(\AfGr_n)$ has a basis of affine Schubert classes, indexed by the set $\mathcal{P}^k$ of {\it $k$-bounded partitions} consisting of those partitions having parts no larger than $k$.  Let $\Lambda$ denote the ring of symmetric functions over $\ZZ$ in the variables $(x) =(x_1, x_2, \dots)$, and denote by $h_i(x)$ the \emph{homogeneous symmetric function} of degree $i$. By \cite{Bott}, there is an isomorphism $\HH_*(\AfGr_n) \cong \Lambda_n$, where $\Lambda_n$ denotes the subring of $\Lambda$ generated by $h_i(x)$ for $0 \leq i \leq k$.  Theorem 7.1 of \cite{LamJAMS} then says that the Schubert classes are represented by the {\it $k$-Schur functions} of \cite{LMktab} under Bott's isomorphism.

Given a $k$-bounded partition $\lambda \in \mathcal P^k$, the corresponding $k$-Schur function is denoted by $s_{\lambda}^{(k)}$.  For $\lambda, \mu \in \mathcal P^k$, the $k$-\emph{Littlewood-Richardson coefficients}  $C_{\lambda, \mu}^{\eta,(k)}$, also called
\emph{affine Littlewood-Richardson coefficients}, are defined by 
\begin{equation*}
\label{klr}
s_\lambda^{(k)}\,s_\mu^{(k)} = 
\sum_{\eta}
C_{\lambda, \mu}^{\eta,(k)}
\,
s_{\eta}^{(k)}
\,,
\end{equation*}
where the sum ranges over $\eta \in\mathcal P^{k}$. Note that $C_{\lambda, \mu}^{\eta, (k)} = 0$ unless $|\mu|+|\lambda|=|\eta|$. 
Whenever $\ell(\lambda)+\lambda_1\leq n$, the $k$-Schur function $s_\lambda^{(k)}$ 
equals the usual Schur function $s_\lambda$ by~\cite[Prop.~39]{LMktab}, and so
 $C_{\lambda, \mu}^{\eta,(k)}$ equals the classical Littlewood-Richardson coefficient $c_{\lambda, \mu}^{\eta}$ 
in this case.

\subsubsection{Quantum-to-affine correspondence for the Grassmannian}\label{sec:LucJenGr}

The Peterson isomorphism implies that all quantum Littlewood-Richardson coefficients arise as affine ones \cite{ls}. As such, the $k$-Schur functions provide a powerful tool for quantum Schubert calculus. In Theorem \ref{thm:LucJenGr}, we review the correspondence $\LJGr$ of Lapointe and Morse which specifies precisely which quantum and affine Littlewood-Richardson coefficients coincide in the case of the Grassmannian.

Given partitions $\mu\subseteq \lambda$, the {\it skew shape} $\lambda/\mu$ is the set of cells which are in $\lambda$, but not in $\mu$.  A skew shape is {\it connected} provided that any cells which share a vertex also share a full edge. 
An {\it $n$-rim hook} is a connected skew shape
which contains $n$ cells, but does not contain any $2 \times 2$ subdiagrams.  
The {\it head} of an $n$-rim hook is its southeasternmost cell.
Given $\lambda \subseteq R_r$ and any $d \in \NN$, we define $\lambda \oplus d$ to be the partition 
in $\mathcal P^k$ obtained by adding $d$ different $n$-rim hooks to $\lambda$ so that all $d$ heads lie in column $r$.

\begin{theorem}\cite[Theorem 18]{lm}\label{thm:LucJenGr}
For any partitions $\lambda, \mu, \nu \subseteq R_r$ and any $d \in \ZZ_{\geq0}$ such that $|\lambda| + |\mu| = |\nu| + nd$, we have
\begin{equation*}
c_{\lambda,\mu}^{\nu,d}\xlongequal{\LJGr}C_{\lambda, \mu}^{\nu\oplus d, (k)}.
\end{equation*}
\end{theorem}

\begin{example}\label{ex:LJGrfav}
To illustrate the Main Theorem, we now instead apply Theorem \ref{thm:LucJenGr} to the original quantum Littlewood-Richardson coefficient from Example \ref{ex:SDfav}.  Recall that $n=k+1=5$ and $r=2$, and consider $\lambda = (2,2,1), \mu = (1,1,0), \nu = (2,0,0) \subset R_2$. When $d=1$, we have
\[
\nu\oplus 1=
{\tiny\tableau[scY]{\times|\times| \times|\times ,\times|\tf,\tf}},
\]
where the cells of the original shape appear in bold, and we add the single 5-rim hook indicated by the cells containing an $\times$, with the head lying in column $r=2$.
Theorem \ref{thm:SDPost} then says that 
\[ c_{\,\tableau[pcY]{|,|,}\, , \, \tableau[pcY]{||}}^{\, \tableau[pcY]{,}\, , \,1} \xlongequal{\LJGr} 
C_{\,\tableau[pcY]{|,|,}\, , \, \tableau[pcY]{||}}^{\nu \oplus 1\, , \, (k)}. \]
\end{example}

\subsubsection{Quantum-to-affine correspondence for the flag variety}\label{sec:LJFl}

Every quantum Littlewood-Richardson coefficient for the flag variety also arises as an affine one, as a direct consequence of the Peterson isomorphism; see \cite[Corollary 9.3]{ls}. We conclude this section by reviewing a combinatorial interpretation of this fact which follows from \cite[Theorem 1.1]{lstoda}, rephrased in a special case as $\LJFl^{-1}$ in Theorem \ref{thm:lstodarephrase} below, following the treatment in \cite[Sec.~7]{BMPSadvances}.

Recall that $R_i$ is the partition defined as $R_i = (i^{n-i})$, and note that for $i \in [k]$, the rectangle $R_i$ is a $k$-bounded partition, which we refer to as a \emph{$k$-rectangle}.  Given $\lambda \in \mathcal P^k$, denote by $\lambda \cup R_i$ the weakly decreasing arrangement of the parts of $\lambda$ and the parts of $R_i$.  Conversely, if $\mu \in \mathcal P^k$ can be written as $\mu = \lambda \cup R_i$ for some $\lambda \in \mathcal P^k$, we say that the $k$-rectangle $R_i$ is \emph{removable} from $\mu$, and the result of \emph{removing the $k$-rectangle} $R_i$ from $\mu$ is the partition $\lambda$. Given any $\mu \in \mathcal P^k$, there is a unique \emph{irreducible} $k$-bounded partition, denoted $\mu^{\downarrow}$, obtained from $\mu$ by removing as many $k$-rectangles as possible.

Given a permutation $w=[w_1 \cdots w_n] \in S_n$, its \emph{inversion sequence} $\Inv(w)$ is defined by $\Inv_i(w) := \# \{ j>i : w_i > w_j \}$. Following \cite{BMPSadvances}, we define an injection
\begin{align}
 \tilde \lambda: &\ S_n \to \mathcal P^k \quad\quad \text{by}\nonumber \\
& \ w\  \mapsto \zeta(w)',\ \ \text{where}\nonumber \\
&\  \zeta_i(w) :=  \Inv_{i}(w_0 w) + \binom{n-i}{2}, \ \ \text{for} \ \  i \in [k].\label{eq:perm2part}
\end{align}
For any $w\in S_n$, set $\tilde \lambda_w := \tilde \lambda (w)$, similar to our notation
for associating a Grassmann permutation $u$ with the partition $\lambda_u$.
The \emph{descent set} of $w \in S_n$ is defined as $D(w) = \{ i \in [k] \mid w_i > w_{i+1}\}$, and the \emph{descent vector} of $w$ is defined to be $\D(w) = \sum_{i \in D(w)} \varepsilon_i$, where $\varepsilon_i \in \ZZ^k$ denotes the $i^{\text{th}}$ standard basis vector; we set $\varepsilon_0 = \varepsilon_n = 0$ by convention. Given any $\d =(d_1, \dots, d_k) \in \ZZ^k$, we define an associated vector by
\begin{equation}\label{eq:tilded}
\tilde \d := \sum_{i \in [k]} d_i(\varepsilon_{i-1} -2\varepsilon_i + \varepsilon_{i+1}) \in \ZZ^k.
\end{equation}

\begin{theorem}[\cite{lstoda, BMPSadvances}]
\label{thm:lstodarephrase}
For any permutations $u,v,w\in S_n$ and any $\d \in\mathbb Z_{\geq 0}^{k}$ such that 
$\tilde \d= \D(w)-\D(u)-\D(v),$
we have 
\begin{equation*}
c_{u,v}^{w,\mathbf d} \xlongequal{\LJFl^{-1}} C_{\tilde \lambda_u^{\downarrow},\, \tilde \lambda_v^{\downarrow}}^{\tilde \lambda_w^{\downarrow},(k)}.
\end{equation*}
\end{theorem}

As the notation suggests, the correspondence presented in Theorem \ref{thm:lstodarephrase} is inverse to $\LJFl$ from the Main Theorem.  We elect to present here the version of this relationship which is both somewhat familiar from the literature on the Peterson isomorphism, and requires fewer technicalities to formally state.  See Theorem \ref{thm:LJFlcomb} in Section \ref{sec:AffComp} for the precise definition of $\LJFl$.

\begin{example}\label{ex:LJFlfav}
To conclude our illustration of the Main Theorem, we now apply Theorem \ref{thm:lstodarephrase} to the quantum Littlewood-Richardson coefficient from Example \ref{ex:FlTfav} illustrating $\FlT$.  Denote the resulting permutations in $S_5$ by $u=[13245], v=[25134], w=[21534],$ and set $\d=(0,1,0,0)$. We see that $\D(u)=\D(v) = \varepsilon_2$ and $\D(w) = \varepsilon_1+\varepsilon_3$, while $\tilde \d = \varepsilon_1-2\varepsilon_2+\varepsilon_3$ by \eqref{eq:tilded}, so that indeed $\tilde \d = \D(w)-\D(u)-\D(v)$.

We illustrate the construction $\tilde \lambda_w^{\downarrow}$ on the permutation $w = [21534]$ in detail.  Compute that $w_0w = [45132]$, in which case $\Inv(w_0w) = (3,3,0,1)$, and so $\tilde \lambda(w) = \zeta(w)'$ where $\zeta(w) = (3,3,0,1)+(6,3,1,0)=(9,6,1,1)$ by formula \eqref{eq:perm2part}.  
Transposing, $\tilde \lambda_w = (\red{4},\blue{2,2,2},2,2,1,1,1) = (2,2,1,1,1) \cup R_4 \cup R_2\in\mathcal P^4$, where
we have colored the parts of the removable $k$-rectangles $\red{R_4}$ and $\blue{R_2}$ in red and blue, respectively.  
Since no further $k$-rectangles are removable, we have $\tilde \lambda_w^{\downarrow} = (2,2,1,1,1) \in \mathcal P^4$.  
In like manner, $\tilde \lambda_u^{\downarrow} = (2,2,1)$ and $\tilde \lambda_v^{\downarrow} = (1,1)$.  Theorem \ref{thm:lstodarephrase} then says that
\begin{equation}\label{eq:LJFlbackward}
c_{[13245]\, ,\, [25134]}^{[21534]\, ,\, (0,1,0,0)}  \xlongequal{\LJFl^{-1}} C_{\,\tableau[pcY]{|,|,}\, , \, \tableau[pcY]{||}}^{\,\tableau[pcY]{|||,|,} \, , \, (4)}.
\end{equation}
In particular, observe that the resulting affine Littlewood-Richardson coefficient is identical to the output from Example \ref{ex:LJGrfav} illustrating $\LJGr$.
\end{example}

\subsection{Illustrating the Main Theorem}\label{sec:favex}

To conclude this section, we unify the examples used to illustrate the results in a single diagram, in order to more clearly demonstrate  the Main Theorem. Given the quantum Littlewood-Richardson coefficient $c_{\lambda, \mu}^{\nu, d}$ for $\QH^*(\Gr_{3,5})$ indexed by $\lambda = (2,2,1), \mu = (1,1,0), \nu = (2,0,0)$ and $d=1$, the following quantum and affine Littlewood-Richardson coefficients are equal:

\begin{center}
\begin{tikzcd}[row sep=tiny]
& c_{\tableau[pcY]{|}\ , \ \tableau[pcY]{||,}}^{\emptyset\, , \, 1}  \arrow[r,"\PC"] & c_{[12435]\, ,\, [23514]}^{[23154]\, ,\, (0,0,1,0)} \arrow[dd, rightarrow, "\FlT"]  \\
 c_{\,\tableau[pcY]{|,|,}\, , \, \tableau[pcY]{||}}^{\, \tableau[pcY]{,}\, , \,1}  \arrow[ur, rightarrow, "\SD"] \arrow[dr,"\LJGr"'] & &  \\
& C_{\,\tableau[pcY]{|,|,}\, , \, \tableau[pcY]{||}}^{\,\tableau[pcY]{\fl| \fl| \fl|\fl,\fl|,} \, , \, (4)} \arrow[r, leftarrow, "\LJFl^{-1}"] & c_{[13245]\, ,\, [25134]}^{[21534]\, ,\, (0,1,0,0)}
\end{tikzcd}
\end{center}

Note that we have visualized each equality in the diagram above using an arrow to indicate the direction in which we have chosen to apply the stated correspondence in this particular example.  Many of these relationships are reversible, of course, the most important of which for our purposes is $\LJFl^{-1}$;  see Theorem \ref{thm:LJFlcomb} in Section \ref{sec:AffComp}, and compare Examples \ref{ex:LJFlfav} and \ref{ex:LJFlcombfav}.


\section{The Peterson comparison formula for the Grassmannian}
\label{sec:comparison}

The goal of this section is to prove Theorem \ref{thm:comparison} by specializing the Peterson comparison formula proved by Woodward \cite{w} to the case of the type $A_k$ Grassmannian.  After reviewing the required root system preliminaries in Section \ref{sec:roots}, we state the type-free version of the Peterson comparison formula in Section \ref{sec:PCtypefree}.  We then specialize to the context of the Grassmannian $\Gr_{m,n}$ in Section \ref{sec:PCGr}, where we prove Theorem \ref{thm:comparison}.

\subsection{Root system preliminaries}\label{sec:roots}

Let $G$ be a connected, simply connected, semisimple complex reductive group of rank $k$.  Fix a Borel subgroup $B$ and a split maximal torus $T$, and denote the Weyl group by $W$.  The Weyl group is a Coxeter group $(W,S)$, where the generators are denoted by $s_i \in S$.  Fix a Cartan subalgebra $\mathfrak{h}$ of the Lie algebra $\mathfrak{g} = \operatorname{Lie}(G)$.  Denote by $R$ the set of roots, and denote by $R^+$ those roots which are positive with respect to $B$. Fix an ordered basis $\Delta = \{ \alpha_i\}_{ i = 1}^{k}$ of simple roots in $\mathfrak{h}^*$, each of which corresponds to a unique simple reflection $s_i \leftrightarrow \alpha_i$.  There is a basis $\Delta^\vee=\{\alpha_i^\vee \}$ for $R^\vee$ of simple coroots in $\mathfrak{h}$, and these bases are dual with respect to the pairing $\langle \cdot , \cdot \rangle : \mathfrak{h} \times \mathfrak{h}^* \to \ZZ$.  Denote the coroot lattice by $Q^\vee = \bigoplus \ZZ \alpha_i^\vee$.  For an element $\gamma = \sum c_i\alpha^\vee_i \in Q^\vee$, we define the support of $\gamma$ to be the subset $\operatorname{Supp}(\gamma) \subseteq [k]$ of indices $i$ such that $c_i \neq 0$. 

Denote by $P$ a standard parabolic subgroup of $G$, equivalently $P \supseteq B$, and recall that the standard parabolics are in bijection with subsets $\Delta_P \subseteq \Delta$. We denote the corresponding coroot lattice by $Q_P^\vee = \bigoplus_{\alpha_i \in \Delta_P} \ZZ \alpha_i^\vee$. The positive roots for $P$, denoted by $R_P^+$, are those roots in $R^+$ which are formed as nonnegative linear combinations of the roots in $\Delta_P$.  
The Weyl group $W_P$ for the parabolic $P$ is generated by those simple reflections $s_i$ which correspond to the roots in $\Delta_P$, and its longest element is denoted by $w_0^P$.  The set of minimal length coset representatives in the quotient $W/W_P$ will be denoted by $W^P$.

\subsubsection{Maximal parabolic subgroups in type $A_k$}

In this paper, we focus on the case in which $G=\SL_n$, the standard Borel is the subgroup of upper-triangular matrices, and the torus is the subgroup of diagonal matrices.  In this case, the Weyl group is the symmetric group $S_n$, and each simple transposition $s_i$ for $ i \in [k]$ can be identified with a reflection in the $k$-dimensional subspace  $V = \{ \vec{v} \in \mathbb{R}^n \mid \sum_{i=1}^n v_i = 0\}$ of $\mathbb{R}^n$.  In particular, the generator $s_i$ reflects across the hyperplane orthogonal to the simple root $\alpha_i = e_i - e_{i+1} \in \Delta$, which corresponds to acting on $\vec{v} \in \mathbb{R}^n$ by interchanging entries $v_i$ and $v_{i+1}$.  The positive roots are then of the form $\alpha_{ij} := e_i - e_j$ where $i < j$.  Since $G = \SL_n$ is simply-laced, the roots and coroots coincide, and so the basis of simple coroots is again given by $\alpha_i^\vee = e_i-e_{i+1}$.  In addition, the pairing $\langle \cdot , \cdot \rangle$ between coroots and roots is the standard Euclidean inner product.

For the majority of Section \ref{sec:comparison}, we further specialize to the situation in which $P$ is a maximal parabolic subgroup of $G=\SL_n$.  In this case, $\Delta_P = \Delta \backslash \{ \alpha_m \} $ for a single index $m \in [k]$, and the quotient $G/P \cong \Gr_{m,n}$.  The Weyl group $W_P \cong S_m \times S_r$, which has  longest element  $w_0^P = [ m \cdots 1\mid n \cdots m+1]$.  The positive roots for $P$ are given by 
$$R_P^+ = \{ e_i-e_j \mid 1 \leq i < j \leq m-1\ \text{or}\ m+1 \leq i < j < n\},$$ so that the quotient $Q^\vee/Q_P^\vee \cong \ZZ \alpha_m^\vee$. The minimal length coset representatives in $W^P$ are those permutations in $\mGrass$, and are thus in bijection with partitions in the rectangle $R_r$.

\subsection{The Peterson comparison formula for $G/P$}\label{sec:PCtypefree}

We now review the Peterson comparison formula relating the quantum Littlewood-Richardson coefficients for any homogeneous space $G/P$ to certain ones for $G/B$, with the goal of specializing to the case of the Grassmannian and providing a proof of Theorem \ref{thm:comparison}.  We largely follow the treatment of Lam and Shimozono \cite{ls}, though this formula was originally stated by Peterson \cite{Pet} and proved by Woodward \cite{w}.

To any vector $\mathbf{d} = (d_1, \dots, d_{k}) \in \mathbb{Z}^{k}$ , we can associate a unique coroot $\gamma_{\mathbf{d}} = d_1\alpha_1^\vee + \cdots + d_{k}\alpha_{k}^\vee \in Q^\vee,$ and vice versa.  Since $H_2(G/B) \cong Q^\vee$, the degree of any curve on $G/B$ can be associated with an element of $Q^\vee$, and we use this isomorphism to identify a degree with a coroot $\gamma_{\mathbf{d}} \in Q^\vee$ or a vector $\mathbf{d} \in \ZZ^{k}$ interchangeably.  Similarly, $H_2(G/P) \cong Q^\vee / Q^\vee_P$, and so the degree of a curve on $G/P$ can be identified with an element in the quotient $\gamma_{\mathbf{d}_P} \in Q^\vee/ Q^\vee_P$, or interchangeably with a vector $\mathbf{d}_P \in \ZZ^{p},$ where $p = \# \Delta \backslash \Delta_P $.

We now state the Peterson comparison formula as it appears in Theorem 10.15 from \cite{ls}, originally proved in Lemma 1 and Theorem 2 of \cite{w}.

\begin{theorem}\cite{w,ls}
\label{T:Woodward}
 Fix a parabolic subgroup $P$, and denote by  $\pi_P: Q^\vee \rightarrow Q^\vee/Q^\vee_P$ the natural projection.
\begin{enumerate}
\item For every $\gamma_{\mathbf{d}_P} \in Q^\vee/Q^\vee_P$, there exists a unique $\gamma_{\mathbf{d}_B} \in Q^\vee$ such that both $\pi_P(\gamma_{\mathbf{d}_B}) = \gamma_{\mathbf{d}_P}$ and $\langle \gamma_{\mathbf{d}_B}, \alpha \rangle \in \{ 0,-1\}$ for all $\alpha \in R^+_P$.
\item For any $u, v, w \in W^P$ and a fixed degree $\mathbf{d}_P$, we have 
\begin{equation*}
c_{u,v}^{w_0ww_0^P,\, \d_P} = c_{u,v}^{w_0ww_0^{P'},\, \d_B}
\end{equation*}
where $P'$ is the parabolic subgroup defined by $\Delta_{P'} = \{ \alpha \in \Delta_P \mid \langle \gamma_{\mathbf{d}_B}, \alpha \rangle = 0\}.$
\end{enumerate}
\end{theorem}

Woodward's original proof of Theorem \ref{T:Woodward} uses the geometry of principal bundles over algebraic curves arising in \cite{AtiyahBott, r}, but these arguments unfortunately do not provide an explicit means for directly computing $\mathbf{d}_B$ or $w_0^{P'}$. We illustrate Theorem \ref{T:Woodward} in the following example, which shows that, in principle, one can calculate the vector $\mathbf{d}_B$ and the element $w_0^{P'}$ using this version of the theorem.

\begin{example}\label{Ex:Woodward}
Suppose that $G/P \cong \Gr_{3,5}$, which means that $P$ is the maximal parabolic in $G=\SL_5$ corresponding to $\Delta_P = \{\alpha_1, \alpha_2, \alpha_4\}$. 
From Example \ref{ex:SDfav}, we know that $c_{\lambda_u, \lambda_v}^{\lambda_w, 1} = 1$, where $\lambda_u=(1,0,0), \lambda_v = (2,1,1), \lambda_w = (0,0,0)$ are the partitions in $R_2$ corresponding to the Grassmann permutations $u=[12435], v=[23514], w=[12345]$ in $S_5^3$.

Any lift of the degree $\mathbf{d}_P = 1$, which corresponds to the coroot $\gamma_{\mathbf{d}_P} = \alpha_3^\vee \in Q^\vee/Q^\vee_P$,  is necessarily of the form $\gamma_{\mathbf{d}_B} = d_1\alpha_1^\vee + d_2\alpha_2^\vee + 1\alpha_3^\vee + d_4\alpha_4^\vee$.  For this parabolic, $R_P^+ = \{\alpha_1, \alpha_2, \alpha_1+\alpha_2, \alpha_4\}$, and we can compute directly that 
\[ \langle \alpha_3^\vee, \alpha_1 \rangle = 0 \quad \text{and} \quad \langle \alpha_3^\vee, \alpha_2 \rangle = \langle \alpha_3^\vee, \alpha_1+\alpha_2 \rangle = \langle \alpha_3^\vee, \alpha_4 \rangle = -1.
\]
Therefore, in fact $\gamma_{\mathbf{d}_B} = \alpha^\vee_3$ itself satisfies $\langle \gamma_{\mathbf{d}_B}, \alpha\rangle \in \{0,-1\}$ for all $\alpha \in R_P^+$. By the uniqueness in part (1) of Theorem \ref{T:Woodward}, we must then have $\mathbf{d}_B = (0,0,1,0)$.  

The same calculation illustrates that $\Delta_{P'} = \{\alpha_1\}$ so that the longest element of the parabolic subgroup $W_{P'} = \langle s_1 \rangle$ equals $w_0^{P'} = s_1$.  Therefore, by part (2) of Theorem \ref{T:Woodward}, we have
\[ c_{\tableau[pcY]{|}\ , \ \tableau[pcY]{||,}}^{\emptyset, 1}  = \,c_{[12435]\, ,\, [23514]}^{[23154]\, ,\, (0,0,1,0)},\]
where the quantum Littlewood-Richardson coefficient on the righthand side is for the complete flag variety $\Fl_5$; compare Example \ref{ex:PCfav} obtained via $\PC$.
\end{example}

\subsection{The Peterson comparison formula for the Grassmannian}\label{sec:PCGr}

Any individual calculation comparing a particular quantum Littlewood-Richardson coefficient for $G/P$ to one from $G/B$ using the Peterson comparison formula from Theorem \ref{T:Woodward} can be carried out in a manner similar to Example \ref{Ex:Woodward}.  Theorem \ref{thm:comparison} provides a closed formula for both the degree $\mathbf{d}_B$ and the element $w_0^{P'}$ occurring in Theorem \ref{T:Woodward}, in the special case where $P$ is a maximal parabolic subgroup of $G = \SL_n$.  We remark that formulas for both $\d_B$ and $w_0^{P'}$ in this case were obtained by Leung and Li, though using somewhat different language; see \cite{LLiJDG, LLiIMRN}.

For the remainder of Section \ref{sec:comparison}, we specialize to the situation of $G = \SL_n$ and $P$ the maximal parabolic subgroup such that $\Delta_P = \Delta \backslash \{ \alpha_m \} $ for some $m\in [k]$.  For $G/P \cong \Gr_{m,n}$, a degree $\mathbf{d}_P$ is a single nonnegative integer, and so we typically omit both the subscript and vector notation and write $d$ for the degree.

We now define some auxiliary notation which will be useful in describing the parabolic subgroup arising in part (2) of Theorem \ref{T:Woodward}. Denote by 
\begin{equation*}
\Delta_{P_{ij}} := \{ \alpha_i, \alpha_{i+1}, \dots, \alpha_{j}\},
\end{equation*}
where by convention $\Delta_{P_{ij}} = \emptyset$ if $i> j$.  We then have the corresponding Weyl group
\begin{equation*}
W_{P_{ij}} = \langle s_i, s_{i+1}, \dots, s_{j} \rangle,
\end{equation*}
where $W_{P_{ij}}$ is trivial if $i> j$. We can express the window for the longest element in $W_{P_{ij}}$ maximizing the number of inversions as follows:
\begin{equation}\label{E:w0Pij}
w_0^{P_{ij}} = [ 1 \cdots i-1\ |\ j+1 \cdots i\ |\ j+2 \cdots n].
\end{equation}
Of course, if $W_{P_{ij}}$ is trivial, then $w_0^{P_{ij}}$ is simply the identity in $S_n$.

We can now state the key proposition from which Theorem \ref{thm:PetComp} follows; compare \cite[Lemma 3.6]{LLiJDG} and \cite[p. 3718]{LLiIMRN}. Since the parabolic subgroup $P'$ in Theorem \ref{T:Woodward} depends on the fixed degree $d$, we indicate this dependence by denoting $P_{\!d}':=P'$ henceforth.

\begin{proposition}\label{T:WoodwardGr}
Suppose that $\Delta_P = \Delta \backslash \{ \alpha_m \}$, and fix an integer $d$ such that $0 \leq d \leq \min \{m,r\}$.  Then
\begin{equation}\label{E:dB}
\textbf{d}_B = (0^{m-d}, 1, 2, \dots, d-1, d, d-1, \dots, 2, 1, 0^{r-d}).
\end{equation}
 In addition, $\Delta_{P'_d} = \Delta \backslash \{\alpha_{m-d}, \alpha_m, \alpha_{m+d}\}$, and so we can write
\begin{equation}\label{eq:Coxprod}
w_0^{P'_d} = w_0^{P_{1, m-d-1}}\cdot w_0^{P_{m-d+1, m-1}}\cdot w_0^{P_{m+1,m+d-1}}\cdot w_0^{P_{m+d+1,k}}.
\end{equation}
\end{proposition}

Before proving a proof, we provide an example which illustrates how to easily compute the values of $\d_B$ and $w_0^{P'_d}$ using Proposition \ref{T:WoodwardGr}.

\begin{example}\label{Ex:Gr(49)}
Consider $\Gr_{4,9}$ so that $m=4$ and $r=5$.  We use Proposition \ref{T:WoodwardGr} to compute both $\mathbf{d}_B$ and $w_0^{P'_d}$ for any degree $0 \leq d\leq 4$.   The degree $\d_B$ has a palindromic nature, and is easily derived by placing the value $d$ in the $m^{\text{th}}$ entry and decreasing in both directions.

\begin{table}[htp]
\begin{center}\label{Tab:Gr(49)WoodwardEx}
\begin{tabular}{|c|c|c|} \hline
$d$ & $\mathbf{d}_B$ & $w_0^{P'_d}$ \\ \hline
 
0 & $(0,0,0,0,0,0,0,0)$ & $[4 3 2 1 \, \textcolor{red}{|}\textcolor{blue}{|}\textcolor{red}{|} \, 98765] $\\ \hline

1 & $(0,0,0,1,0,0,0,0)$ & $[3 2 1\, \textcolor{red}{|}\, 4\, \textcolor{blue}{|}\, 5\, \textcolor{red}{|}\, 9 8 7 6]$  \\ \hline

2 & $(0,0,1,2,1,0,0,0)$& $[2 1\, \textcolor{red}{|}\, 4 3\, \textcolor{blue}{|}\, 6 5\, \textcolor{red}{|}\, 9 8 7]$  \\ \hline

3 &$(0,1,2,3,2,1,0,0)$ & $[1\, \textcolor{red}{|}\, 4 3 2\, \textcolor{blue}{|}\, 7 6 5\, \textcolor{red}{|}\, 9 8]$  \\ \hline

4 &$(1,2,3,4,3,2,1,0)$ & $[\textcolor{red}{|}\, 4 3 2 1\, \textcolor{blue}{|}\, 9 7 6 5\, \textcolor{red}{|}\, 9]$ \\ \hline
\end{tabular}
\end{center}
\end{table}
\end{example}

Note that each of the elements $w^{P_{ij}}_0$ occurring in \eqref{eq:Coxprod} acts on a disjoint subset of $[n]$. We may thus record the window for the permutation $w_0^{P'_d}$ using the following algorithm:
\begin{enumerate}
\item Draw the window for the identity permutation.
\item Draw one blue line at position $m$; i.e. between $m$ and $m+1$.
\item Draw two red lines at positions $m\pm d$; note that $0 \leq m\pm d \leq n$ since by hypothesis $0 \leq d \leq \min\{m,r\}$, which means that both of these red lines fit in the window.
\item Reverse the numbers within each of the groups separated by these colored lines. 
\end{enumerate}

For the proof of Proposition \ref{T:WoodwardGr}, it will be helpful to first calculate several of the inner products which will arise in the course of the proof, which we record in the following lemma.

\begin{lemma}\label{T:InnerProdCalc}
Fix an integer $d$ such that $0 < d \leq \min\{m, r\}$.  If $\mathbf{d}_B$ is the palindromic vector defined in  \eqref{E:dB}, then $\gamma_B:= \gamma_{\mathbf{d}_B}$ satisfies 
\begin{equation}\label{E:LambdaValues}
\langle \gamma_B, \alpha_j \rangle = 
\begin{cases}
2, & \text{if}\ j = m, \\
-1, & \text{if}\ j = m \pm d,\\
0, & \text{otherwise}.
\end{cases}
\end{equation} 
In particular, $\langle \gamma_B, \alpha_j \rangle \in \{0,-1\}$ for all $\alpha_j \in \Delta_P$.
\end{lemma}

\begin{proof}
Start by expressing $\gamma_B$ explicitly as 
\begin{equation}\label{E:LambdaBForm}
\gamma_B = \alpha^\vee_{m-d+1} + 2\alpha^\vee_{m-d+2} + \cdots + (d-1)\alpha^\vee_{m-1} + d\alpha^\vee_m + (d-1)\alpha^\vee_{m+1} + \cdots + 2\alpha^\vee_{m+d-2} + \alpha^\vee_{m+d-1}.
\end{equation}
We thus see that 
\begin{equation*} \operatorname{Supp}(\gamma_B) = \{m-d+1, \dots, m+d-1\}.\end{equation*}
Note that the hypothesis $0 < d \leq \min\{m,r\}$ guarantees that $m-d+1 \leq m+d-1$, so that $\operatorname{Supp}(\gamma_B)$ is a set containing exactly $2d-1$ consecutive integers in $[k]$.  

We now fix any $\alpha_j \in \Delta$ and directly compute $\langle \gamma_B, \alpha_j \rangle$.  Recall that in type $A_k$ we have 
\begin{equation}\label{E:Cartan}
\langle \alpha^\vee_i, \alpha_j \rangle = 
\begin{cases}
2, & \text{if}\ i=j, \\
-1, & \text{if}\ |i-j| = 1, \\
0, & \text{otherwise}.
\end{cases}
\end{equation}
The proof thus naturally divides into cases, depending on the relationship between $j$ and $\operatorname{Supp}(\gamma_B)$.

\textbf{Case (1):} $j<m-d$.  Since $j \leq m-d-1$, then for any $i \in \operatorname{Supp}(\gamma_B)$, we have $|i-j| \geq 2$.  Therefore, by linearity of the pairing and Equation \eqref{E:Cartan} we have $\langle \gamma_B, \alpha_j \rangle = 0$.

\textbf{Case (2):} $j = m-d$.  If $j=m-d$, then there exists a unique $i  \in \operatorname{Supp}(\gamma_B)$ such that $|i-j|=1$; namely $i = m-d+1$, for which $|i-j|=1$. For all other $p \in \operatorname{Supp}(\gamma_B)$, we have $|p-j| \geq 2$, and so these each contribute 0 to $\langle \gamma_B, \alpha_j \rangle$.  By Equation \eqref{E:Cartan}, we then have $\langle \gamma_B, \alpha_j \rangle = -1$.

\textbf{Case (3):} $m-d+1 \leq j \leq m-1$.  Here it is natural to further divide into two subcases.  First suppose that $j=m-d+1$, in which case there are exactly two indices $i \in \operatorname{Supp}(\gamma_B)$ such that $|i-j|\leq 1$; namely $i=m-d+1$ and $i=m-d+2$.  Therefore, by linearity and Equation \eqref{E:Cartan} applied to \eqref{E:LambdaBForm}, we have $\langle \gamma_B, \alpha_j \rangle = 1(2) + 2(-1) = 0$ when $j=m-d+1$.

Now suppose that $m-d+1 < j \leq m-1$.  For all such $j$, there are exactly three values for $i \in \operatorname{Supp}(\gamma_B)$ such that $|i-j|\leq 1$; namely $i \in \{j-1, j, j+1\}$.  The coefficients of the coroots $\alpha_{j-1}^\vee, \alpha_j^\vee,$ and $\alpha_{j+1}^\vee$ in $\gamma_B$ are always three consecutive increasing positive integers $a-1, a, a+1$ for some $a \in \{2, \dots, d-1\}$. Combining Equations \eqref{E:LambdaBForm} and \eqref{E:Cartan}, we have $\langle \gamma_B, \alpha_j \rangle = (a-1)(-1)+a(2)+(a+1)(-1) = 0$ for all $m-d+1< j \leq m-1$.

\textbf{Case (4):} $j = m$.  In this case, there are only three values for $i \in \operatorname{Supp}(\gamma_B)$ such that $|i-j| \leq 1$; namely $i \in \{m-1, m, m+1\}$.  Therefore, by Equations \eqref{E:LambdaBForm} and \eqref{E:Cartan}, we see that $\langle \gamma_B, \alpha_j \rangle = (d-1)(-1) + d(2) + (d-1)(-1) = 2,$ as claimed.

All remaining cases in which $j>m$ follow by the symmetry of $\gamma_B$.
\end{proof}

We are now prepared to prove Proposition \ref{T:WoodwardGr}, from which the proof of Theorem \ref{thm:comparison} then directly follows.

\begin{proof}[Proof of Proposition \ref{T:WoodwardGr}]  Fix an integer $0 \leq d \leq \min\{m,r\}$, and define the vector $\mathbf{d}_B$ as in Equation \eqref{E:dB}.  We first verify that $d \mapsto \mathbf{d}_B$ under the Peterson comparison formula of Theorem \ref{T:Woodward}. Under the projection map $\pi_P: Q^\vee \rightarrow Q^\vee/Q_P^\vee$ we clearly have $\pi_P(\mathbf{d}_B) = d$, since the value $d$ is placed in the $m^{\text{th}}$ entry of $\mathbf{d}_B$ and thus corresponds to the coefficient of $\alpha_m^\vee$ in $\gamma_B := \gamma_{\mathbf{d}_B}$.  Regarding $\gamma_B$, it thus remains only to check that $\langle \gamma_B, \alpha \rangle \in \{0,-1\}$ for all $\alpha \in R_P^+$.  

We first observe that if $d=0$, then $\mathbf{d}_B = (0,\dots, 0)$ as well, in which case we automatically have $\langle \gamma_B, \alpha \rangle =0$ for all $\alpha \in R_P^+$.  Now suppose that $0 < d \leq \min \{m, r\}$, and recall that Lemma \ref{T:InnerProdCalc} calculates $\langle \gamma_B, \alpha_j \rangle$ for any simple root $\alpha_j \in \Delta_P$.   Suppose now that $\alpha \in R_P^+$ is not a simple root, say $\alpha = \alpha_{ij} = \alpha_i + \cdots + \alpha_j$.  Since $R_P^+$ is spanned by $\Delta_P = \Delta \backslash \{ \alpha_m\}$, then either $1\leq i< j \leq m-1$ or $m+1 \leq i < j \leq k$.  Either way, by Equation \eqref{E:LambdaValues} of Lemma \ref{T:InnerProdCalc}, we see that $\langle \gamma_B, \alpha_{ij} \rangle$ is the sum of at most one $-1$ and 0's otherwise.  Therefore, for any $0 \leq d \leq \min \{m,r\}$, we have shown that $\langle \gamma_B, \alpha \rangle \in \{0, -1 \}$  for any $\alpha \in R_P^+$.  By the uniqueness of $\mathbf{d}_B$ in Theorem \ref{T:Woodward}, we thus have that $d$ corresponds to $\mathbf{d}_B$ under the Peterson comparison formula.

We now verify our formula for $w_0^{P'_d}$, for which we must compute $\Delta_{P'_d} = \{ \alpha \in \Delta_P \mid \langle \gamma_B, \alpha \rangle = 0\}.$  First suppose that $d=0$, in which case $\mathbf{d}_B = (0, \dots, 0)$ so that $\gamma_B = 0 \in Q^\vee$ as well.  Therefore, $\langle \gamma_B, \alpha \rangle = 0$ for all $\alpha \in \Delta_P$, and so $\Delta_{P'_d} = \Delta_P$. Now suppose $0 < d \leq \min \{m,r\}$, in which case Equation \eqref{E:LambdaValues} of Lemma \ref{T:InnerProdCalc} says that 
\begin{equation*}
\Delta_{P'_d} = \{ \alpha \in \Delta_P \mid \langle \gamma_B, \alpha \rangle = 0\}= \Delta_P \backslash \{ \alpha_{m-d}, \alpha_{m+d} \} = \Delta \backslash \{\alpha_{m-d}, \alpha_m, \alpha_{m+d}\}.
\end{equation*}
Moreover, note that this formula also matches $\Delta_{P'_d} = \Delta_P$ when $d=0$. Therefore, for any $0 \leq d \leq \min \{m,r\}$, we see that $W_{P'_d}$ is generated by the simple roots corresponding to $ \Delta \backslash \{\alpha_{m-d}, \alpha_m, \alpha_{m+d}\}$; namely the generators of the parabolic subgroups $W_{P_{1,m-d-1}},$ $W_{P_{m-d+1,m-1}},$ $W_{P_{m+1,m+d-1}},$ and $W_{P_{m+d+1,k}}$.  (Recall that $W_{P_{ij}}$ is defined to be trivial if $i >j$, so in the $d=0$ case we really only have two nontrivial parabolic subgroups, namely $W_{P_{1,m-1}}$ and $W_{P_{m+1,k}}$.  A similar comment applies when $d=\min\{m,r\}$.)
Finally, observe that each of the subgroups $W_{P_{ij}}$ listed above acts on a disjoint subset of $[n]$.  Therefore, the longest element in $W_{P'_d}$ will equal the product of the longest elements in each of these $W_{P_{ij}}$, as claimed in \eqref{eq:Coxprod}. 
\end{proof}

We remark that the hypothesis $0 \leq d \leq \min\{r,m\}$ on the degree $d$ appearing in Theorem \ref{thm:comparison}, which is formally necessary in the proof of the key Proposition \ref{T:WoodwardGr}, does not impose any actual constraints on the quantum Littlewood-Richardson coefficients to which Theorem \ref{thm:comparison} applies.  Indeed, by degree considerations, if $c_{\lambda, \mu}^{\nu,d} \neq 0$ for some $\lambda, \mu, \nu \subseteq R_{r}$ and $d \in \ZZ$, then in fact $0 \leq d \leq \min\{r,m\}$.  

The proof of Theorem \ref{thm:comparison} now follows directly from Proposition \ref{T:WoodwardGr}.

\begin{proof}[Proof of Theorem \ref{thm:comparison}]
Let $u,v,w \in \mGrass$, and fix an integer $0 \leq d \leq \min \{ m,r\}$.  Apply Proposition \ref{T:WoodwardGr} to obtain the stated formula for $\d=\d_B$ in Theorem \ref{thm:comparison}.  Recall from Proposition \ref{T:WoodwardGr} that $w_0^{P'_d}=w_0^{P_{1, m-d-1}}\cdot w_0^{P_{m-d+1, m-1}}\cdot w_0^{P_{m+1,m+d-1}}\cdot w_0^{P_{m+d+1,k}}.$
Applying Equation \eqref{E:w0Pij} to these $w_0^{P_{ij}}$, each of which acts on a disjoint subset of $[n]$, we thus have
\[w_0^{P'_d} 
 = [ m-d \cdots 1\ \textcolor{red}{|}\ m \cdots m-d+1\ \textcolor{blue}{|}\ m+d \cdots m+1\ \textcolor{red}{|}\ n \cdots m+d+1],
\]
which is the stated formula for $w_0^{P'_d}$ in Theorem \ref{thm:comparison}.  The equality of quantum Littlewood-Richardson coefficients via $\PC$ then follows directly from part (2) of Theorem \ref{T:Woodward}. 
\end{proof}


\section{Composing the affine Littlewood-Richardson comparisons}\label{sec:AffComp}

In this section, we formalize the correspondence $\LJFl$ on those affine Littlewood-Richardson coefficients in the image of $\LJGr$. In the Main Theorem, we apply the correspondence $\LJGr$ to quantum Littlewood-Richardson coefficients indexed by partitions of the form $\lambda_u$ for some Grassmann permutation $u \in \mGrass$. The partial inverse to the map $\tilde \lambda^{\downarrow}$ from Section \ref{sec:LJFl}, which we define on Grassmann permutations in Section \ref{sec:invertGrass}, will thus be useful for studying the composition $\LJFl \circ \LJGr$.  As we will see in Section \ref{sec:LJFlcomb}, however, the third $k$-bounded partition indexing the affine Littlewood-Richardson coefficients appearing in the image of $\LJGr$ corresponds to a permutation with two descents.  In Section \ref{sec:2descents}, we thus extend the results from Section \ref{sec:invertGrass} to permutations having two descents.  The main result of this section is then Theorem \ref{thm:LJFlcomb} in Section \ref{sec:LJFlcomb}, which formally defines $\LJFl$ and provides an inverse to the relation $\LJFl^{-1}$ from Theorem \ref{thm:lstodarephrase}.

\subsection{The map $\tilde \lambda^{\downarrow}$ on Grassmann permutations}\label{sec:invertGrass}

Recall the function $\tilde \lambda^{\downarrow}: S_n \rightarrow \mathcal P^k$ from Section \ref{sec:LJFl}, which maps $w \mapsto \tilde \lambda^{\downarrow}_w$ by first transposing the $k$-bounded partition $\zeta(w)$ defined in \eqref{eq:perm2part} to obtain $\tilde \lambda_w$, and then removing all possible $k$-rectangles to yield the irreducible $k$-bounded partition $\tilde \lambda^{\downarrow}_w$.

\begin{lemma}\cite[Lemma 7.5]{BMPSadvances}
\label{le:onedescent}
If $u \in \jGrass$, then 
\begin{equation}\label{eq:onedescent}
\left( \tilde \lambda_u^{\downarrow} \right)'= \big(n-j-\Inv_1(u), \ldots,n-j-\Inv_j(u) \big)\subseteq R_{n-j}.
\end{equation}
\end{lemma}

For any $u \in \jGrass$, comparing Equations \eqref{eq:lampermdef} and \eqref{eq:permlamdef} with $m=j$, one sees that an alternate formula for the partition $\lambda_u \subseteq R_{n-j}$ is given by 
\begin{equation}\label{eq:lampermalt}
\lambda_u = (\Inv_{j}(u),\ldots, \Inv_{1}(u)).
\end{equation}
Momentarily denote by $\lambda := \lambda_u = (\lambda_1, \dots, \lambda_j)$, so that $\lambda_{j-i+1} = \Inv_i(u)$ as in formula \eqref{eq:lampermalt}.  Comparing the righthand sides of Equations \eqref{eq:complement} and \eqref{eq:onedescent} with $m=j$, we thus see that 
 the complement of $\lambda_u$ in $R_{n-j}$ is given by 
$\lambda_u^{\vee_{n-j}}=\left( \tilde \lambda_u^{\downarrow} \right)'.$
We rewrite this expression as
\begin{equation}\label{eq:thetainverse}
\tilde \lambda_u^{\downarrow} = \left(\lambda_u'\right)^{\vee_j},
\end{equation}
where we have used the fact that $(\mu^{\vee_{n-j}})' = (\mu')^{\vee_{j}}$ for any $\mu \subseteq R_{n-j}$.

Although in general many different permutations in $S_n$ can map onto the same $k$-bounded partition via the map $\tilde \lambda^{\downarrow}$, observation \eqref{eq:thetainverse} provides a partial inverse for $\tilde \lambda^{\downarrow}$ when restricted to Grassmann permutations.

\begin{remark}
\label{re:thetainverse}
Transposing \eqref{eq:onedescent} from Lemma \ref{le:onedescent}, the image of the restriction 
\begin{equation*}
\tilde \lambda^{\downarrow}: \jGrass \rightarrow \mathcal P^k
\end{equation*}
is the set of $k$-bounded partitions contained in the rectangle $R_j$. 
Moreover, when restricted to this image, the map $\tilde \lambda^{\downarrow}$ is invertible by  \eqref{eq:thetainverse}.  That is, given any $\eta \in \mathcal P^k$ such that $\eta \subseteq R_j$, 
 the inverse $\varphi_j$ of $\tilde \lambda^{\downarrow}$ is defined by 
\begin{equation}\label{eq:thetaGrass}
\varphi_j\left( \eta \right) =  w_{(\eta^{\vee_j})'} = w'_{\eta^{\vee_j}} \in \jGrass
\end{equation}
using the notation from \eqref{eq:lampermdef}, and where the second equality follows from \eqref{eq:conjugate}.
When the ambient rectangle containing the given $k$-bounded partition $\eta \subseteq R_j$ is understood from context, we simply write $\varphi(\eta)$.
\end{remark}

We demonstrate the inverse relationship of $\tilde \lambda^{\downarrow}$ and $\varphi$ in the example below.

\begin{example}\label{ex:thetainverse}
Consider the permutation $v=[25134] \in S^2_5$ from Example \ref{ex:LJFlfav}, in which case $j=2$ and $n=5$.  To review the construction of $\tilde \lambda^{\downarrow}_v$ from Section \ref{sec:LJFl}, compute via \eqref{eq:perm2part} that $\zeta(v) = (3,0,2,1)+(6,3,1,0)=(9,3,3,1)$ so that $\tilde \lambda_v =\zeta(v)'= (4,3,3,1,1,1,1,1,1) = R_4\cup R_3 \cup R_1 \cup (1,1)$ which then reduces to $\tilde \lambda^{\downarrow}_v =  (1,1,0) \subset R_2$, as claimed in Example \ref{ex:LJFlfav}. Alternatively, computing $\Inv(v) = (1,3)$ and applying Lemma \ref{le:onedescent} with $n-j=3$, we have $\left( \tilde \lambda_v^{\downarrow}\right)' = (3-1,3-3) = (2,0) \subset R_3$, which indeed agrees with our calculation directly from the definition of $\tilde \lambda^{\downarrow}$.

Conversely, given the $4$-bounded partition $\mu = (1,1,0) \subset R_2$, taking the complement yields $\mu^{\vee_2}=(2,1,1)$, and then transposing we have $(\mu^{\vee_2})' = (3,1) \subset R_3$.  Recording the permutation corresponding to the partition $(3,1) \subset R_3$ via \eqref{eq:lampermdef} then gives $w_{(\mu^{\vee_2})'} = [25134] \in S^2_5$.  We thus see that indeed $\varphi(\mu) = v$, illustrating Remark \ref{re:thetainverse}.
\end{example}

\subsection{The map $\tilde \lambda^{\downarrow}$ on permutations with two descents}\label{sec:2descents}

Although the inverse $\varphi$ in Remark \ref{re:thetainverse} is only defined when $\tilde \lambda^{\downarrow}$ is applied Grassmann permutations, in this section we extend the results from Section \ref{sec:invertGrass} to permutations having two descents.  We begin by recording an algorithm for writing a permutation with two descents as the product of two Grassmann permutations; see also \cite[Lemma 2.1]{Richmond12}.

\begin{lemma}\label{le:constructfactors}
Fix a pair of integers $1 \leq a < b < n$. A permutation $w \in S_n$ satisfies $D(w)= \{a,b\}$ if and only if $w=w^2w^1$ for two Grassmann permutations $w^2 \in S^a_n$ and $w^1 \in S^b_n$ such that $w^1(i)=i$ for all $i \in [a]$.
\end{lemma}

\begin{proof}
If $w = [w_1 \cdots w_n] \in S_n$ has exactly two descents, say in positions $1 \leq a<b<n$, then we can write $w$ as a product of two Grassmann permutations $w=w^2w^1$ where $w^2 \in S_n^a$ and $w^1\in S_n^b$, as follows.  The window for $w^2$ is obtained by keeping  entries $w_1, \dots, w_a$ from the window for $w$, and then arranging  the entries $w_{a+1}, \dots, w_n$  into increasing order.  To construct the window for $w^1$, define $w^1_i = i$ for all $1 \leq i \leq a$ so that $w^1=[1 \cdots a\mid w^1_{a+1} \cdots w^1_n]$. Define the remaining $n-a$ entries of $w^1$ using the values $\{a+1, \dots, n\}$, but maintaining the same relative order from the window for $w$. Precisely, 
taking entries $w_{a+1}, \dots, w_n$ from the window for $w$ and re-indexing them in increasing order $w_{i_1} < \cdots < w_{i_{n-a}}$, define $w^1_{i_j}:= a+j$ for all $1 \leq j \leq n-a$. It is then straightforward to verify that $w=w^2w^1$.

Conversely, the given conditions on $w^1\in S^b_n$ imply that $w^1_{a+1} < \cdots <w^1_b$ and $w^1_{b+1} < \cdots < w^1_n$, in addition to the fact that $w^1(i)=i$ for all $i \in [a]$. Since $D(w^2)=\{a\}$ for $a<b$, we find that $D(w^2w^1)=\{a,b\}$.
\end{proof}

We illustrate Lemma \ref{le:constructfactors} with the following example.

\begin{example}\label{ex:w2w1}
Consider the permutation $w=[5 9 2 4 6 7 8 1 3]\in S_9$, which has two descents in positions $a=2$ and $b=7$. Using Lemma \ref{le:constructfactors}, define $w^2 \in S_9^2$ by rearranging the last $9-2=7$ entries of the window for $w$, and so $w^2 = [591234678]$.  To define $w^1 \in S^7_9$, by Lemma \ref{le:constructfactors} we first set $w^1_i = i$ for $i \in \{1,2\}$, and then place the elements of $\{3, \dots, 9\}$ according to the same relative ordering as the last 7 entries of the window for $w$ to obtain $w^1=[124678935]$.  Indeed $w=w^2w^1$, confirming Lemma \ref{le:constructfactors}.
\end{example}

The following lemma from \cite{BMPSadvances} counting the number of parts in the partition $\tilde \lambda^{\downarrow}_w$ will be useful in the proof of Proposition \ref{prop:inversions} below.

\begin{lemma}\cite[Lemma 7.3]{BMPSadvances}
\label{le:pain in the ass}
Let $w\in S_n$. Denote by $\Inv(w_0 w) = (I_1,\dots,I_k)$, and by $n_i$ the number of parts of size $i$ in $\tilde \lambda^{\downarrow}_w$.  Then for any $1 \leq i \leq k-1$, we have
\begin{equation}
\label{eq:multiplicities}
n_i=\begin{cases}
k-i+I_i-I_{i+1} & \text{ if } i \in D(w),\\
I_i-I_{i+1}-1 & \text{ if } i\not\in D(w).
\end{cases}
\end{equation}
\end{lemma}

We are now able to prove the main result of this subsection, which will be critical in the proof of Theorem \ref{thm:LJFlcomb}.

\begin{prop}
\label{prop:inversions}  
Given $1 \leq a < b < n$ and two Grassmann permutations $w^2 \in S_n^a$ and $w^1 \in S_n^b$ such that $w^1(i)=i$ for all $i \in [a]$, define $w=w^2w^1 \in S_n$. Then
 \begin{equation*}
\left( \tilde \lambda_w^{\downarrow}\right) '=\left( \tilde \lambda_{w^2}^{\downarrow}\right) '+\left( \tilde \lambda_{w^1}^{\downarrow}\right) '.
\end{equation*}
\end{prop}

\begin{proof}
Given $w=w^2w^1$ as described, let $\eta = \tilde \lambda^{\downarrow}_w$.  Denote by $n_i$ the number of parts of size $i$ in the partition $\eta$, and note that 
\begin{equation}\label{eq:nitranspose}
n_i = (\eta')_i - (\eta')_{i+1}.
\end{equation}
Since $\eta$ is irreducible, we also have $l(\eta') <k$. Lemma \ref{le:constructfactors} implies that $D(w)=\{a,b\}$ and thus $\Inv_i(w)=0$ for all $i>b$.  
If we denote by $\Inv(w_0 w) = (I_1,\dots,I_k)$, then $I_i=n-i-\Inv_i(w)$ for all $i \in [k]$.  

Combining these observations with Equation \eqref{eq:multiplicities} from Lemma \ref{le:pain in the ass} gives
\begin{align}\label{eq:ni}
n_i=
\left\{ \!\!
\begin{array}{l@{\ =\ }ll}
k-a+I_a-I_{a+1} &\  n-a-\Inv_{a}(w)+\Inv_{a+1}(w)& \text{ if } i = a, \\
k-b+I_b-I_{b+1} & \ n-b-\Inv_{b}(w) & \text{ if } i = b, \\
I_i-I_{i+1}-1 &\ \Inv_{i+1}(w)-\Inv_i(w) & \text{ if } a \neq i<b, \\
I_i-I_{i+1}-1 & \ 0 & \text{ if } i > b.
\end{array}\right.
\end{align}
Combining \eqref{eq:nitranspose} and \eqref{eq:ni} for $i \geq b$, we see that $(\eta')_i = 0$ for all $i > b$, and 
\begin{equation}\label{eq:nb}
(\eta')_b =n-b-\Inv_{b}(w).
\end{equation}
 Comparing \eqref{eq:nitranspose} and \eqref{eq:ni} for $i=b-1$, we see that $ n_{b-1}=(\eta')_{b-1} - (\eta')_{b} =   \Inv_{b}(w)-\Inv_{b-1}(w)$, in which case we can use \eqref{eq:nb} to obtain 
$(\eta')_{b-1} = n-b-\Inv_{b-1}(w).$
Iterating this argument, we obtain
\begin{equation}\label{eq:ni>a}
(\eta')_{j} = n-b-\Inv_{j}(w)
\end{equation}
for all $a+1 \leq j \leq b$.  Now combining \eqref{eq:nitranspose} with $i=a$ and \eqref{eq:ni>a} with $j=a+1$, we have $n_a  =  (\eta')_a - (\eta')_{a+1} =  (\eta')_{a}- n+b+\Inv_{a+1}(w).$  Equivalently, comparing this expression for $n_a$ with \eqref{eq:ni} for $i=a$, we have
\begin{equation*}
(\eta')_{a}=2n-a-b-\Inv_{a}(w).
\end{equation*}
Now iterating this argument instead, we obtain 
\begin{equation}\label{eq:ni<a}
(\eta')_{j}=2n-a-b-\Inv_{j}(w)
\end{equation}
for all $1 \leq j \leq a$.  

Putting Equations \eqref{eq:ni>a} and \eqref{eq:ni<a} together, we have
\begin{equation}\label{eq:eta'}
\eta'=
( 2n-a-b-\Inv_{1}(w),\ldots,2n-a-b-\Inv_{a}(w),
n-b-\Inv_{a+1}(w),\ldots,n-b-\Inv_{b}(w)).
\end{equation}
Since $w^1(i)=i$ for all $i \in [a]$, we also have $\Inv_i(w^1)=0$ for all $i \in [a]$.
 Since $a<b$, we may express \eqref{eq:eta'} as a sum 
\begin{align}
\eta' &=(n-a-\Inv_{1}(w^2),\ldots,n-a-\Inv_{a}(w^2))  \nonumber \\
\phantom{\eta'} & + (n-b-\Inv_{1}(w^1),\ldots,n-b-\Inv_{b}(w^1)). \label{eq:eta'sum}
\end{align}
Lemma~\ref{le:onedescent} applied to \eqref{eq:eta'sum} now says that
$\left( \tilde \lambda_w^{\downarrow}\right) ' =\eta'=\left( \tilde \lambda_{w^2}^{\downarrow}\right) '+\left( \tilde \lambda_{w^1}^{\downarrow}\right) '$, as desired.
\end{proof}

We illustrate Proposition \ref{prop:inversions} on the permutation from Example \ref{ex:w2w1}.

\begin{example}
 Recall from Example \ref{ex:w2w1} that $w=[5 9 2 4 6 7 8 1 3]=w^2w^1$ for $w^2 = [591234678] \in S^2_9$ and $w^1=[124678935] \in S^7_9$.  By Lemma \ref{le:onedescent}, $\left( \tilde \lambda^{\downarrow}_{w^2} \right)' = (7-4,7-7)=(3,0) \subset R_7$ and similarly,  $\left( \tilde \lambda^{\downarrow}_{w^1} \right)' = (2,2,1,0,0,0,0) \subset R_2$. Taking the sum (and omitting all trailing zeros), we have $\left( \tilde \lambda^{\downarrow}_{w^2} \right)' +\left( \tilde \lambda^{\downarrow}_{w^1} \right)'  = (5,2,1) \in \mathcal P^8$. 

Now compute $\left( \tilde \lambda^{\downarrow}_{w} \right)' $ directly from the definition as follows.  Since
$w_0w= [5 1 8 6 4 3 2 9 7]$, then $\Inv(w_0w)=(4,0,5,3,2,1,0,1)$ and so $\zeta(w) = \Inv(w_0w)+(28,21,15,10,6,3,1,0)=(32,21,20,13,8,4,1,1)$.  Transposing and removing the $k$-rectangles $R_1 \cup R_3 \cup R_4 \cup R_5 \cup R_6 \cup R_8$, we obtain $\left( \tilde \lambda^{\downarrow}_{w} \right) =(3,2,1,1,1)$.  Therefore, $\left( \tilde \lambda^{\downarrow}_{w^2} \right)' =(5,2,1) = \left( \tilde \lambda^{\downarrow}_{w^2} \right)' +\left( \tilde \lambda^{\downarrow}_{w^1} \right)'$, confirming Proposition \ref{prop:inversions}.
\end{example}

\subsection{The correspondence $\LJFl$}\label{sec:LJFlcomb}

The goal of this section is to provide an inverse to the correspondence from Theorem \ref{thm:lstodarephrase} on the image of $\LJGr$. We begin by formalizing some notation that will be useful both in Theorem \ref{thm:LJFlcomb} defining $\LJFl$ below, as well as later in Section \ref{sec:proof}. 

Given $\nu \subseteq R_r$ and $d \in \ZZ_{\geq 0}$, 
recall that $\nu \oplus d$ is defined to be the partition obtained by adding $d$ different $n$-rim hooks to $\nu$ such that all $d$ heads lie in column $r$.

\begin{definition}
\label{def:etas} 
Let $\nu \subseteq R_r$ and $d \in \ZZ_{\geq 0}$.
\begin{enumerate}
\item Define $t:=\diag_0((\nu\oplus d)^{\vee_r})$. Note that the partition $\nu \oplus d$ may not be a subset of the rectangle $R_r$, so $(\nu \oplus d)^{\vee_r}$ here denotes the complement of the portion of the shape $\nu \oplus d$ which is contained in $R_r$, by abuse of notation.
\item Subdivide $\nu \oplus d = (\eta_t^1, \eta_t^2)$ into a pair of partitions such that $\eta_t^1$ consists of the bottom $m-t$ rows of $\nu \oplus d$; that is, $l(\eta_t^1) =m-t$. When $t$ is understood, we simply write $(\eta^1,\eta^2)$.
\end{enumerate}
\end{definition}

\begin{remark}
\label{re:diagdist}
For $\nu \subseteq R_r$ and $d \in \ZZ_{\geq 0}$, let 
$\delta:=\diag_0(\nu^{\vee_r})$ and $t:=\diag_0((\nu\oplus d)^{\vee_r})$.
Since $\delta$ is the diagonal distance from the upper right corner of $R_r$ to the boundary of $\nu$,
and $t$ is the distance from the same corner to the boundary of $\nu\oplus d$, we have that $\delta=d+t$ and 
$0 \leq t \leq \delta$.
\end{remark}

The integer $t$ and partition $\nu \oplus d = (\eta^1, \eta^2)$ are illustrated in Examples
\ref{ex:LJFlcombfav} and \ref{ex:nuoplusd} below. 
We are now prepared to state and prove the main result of Section \ref{sec:AffComp}.

\begin{theorem}\label{thm:LJFlcomb}  
For any partitions $\lambda,\mu,\nu\subseteq R_r$ and any $d \in \ZZ_{\geq 0}$ such that
$|\lambda|+|\mu|=|\nu|+nd$, we have
\begin{equation}\label{eq:LJFlcomb}
C_{\lambda, \mu}^{\nu\oplus d, (k)}
\xlongequal{\LJFl} c_{\varphi_r(\lambda),\varphi_r(\mu)}^{\varphi_{r-t}(\eta^2)\varphi_{r+t}(\eta^1),\t}\,,
\end{equation}
where $t=\diag_0((\nu\oplus d)^{\vee_r})$ and $\nu \oplus d=(\eta^1,\eta^2)$, and 
$\t=(0^{r-t},1,2,\ldots,t,\ldots,2,1,0^{m-t}).$
 Moreover, this correspondence is inverse to $\LJFl^{-1}$ from Theorem \ref{thm:lstodarephrase}.
\end{theorem}

\begin{proof}
By definition of $\eta$ and $t$, we have $\eta^1\subseteq R_{r+t}$ and $\eta^2\subseteq R_{r-t}$.  Denote the corresponding Grassmann permutations by $w^1:=\varphi_{r+t}(\eta^1)$ and $w^2:=\varphi_{r-t}(\eta^2)$, where the descents are in positions $r+t$ and $r-t$, respectively, according to Remark \ref{re:thetainverse}. 
By the definition of $t$, the partition $\eta^1 \subseteq R_{r+t}$ has at least $r-t$ columns of full height $m-t$.  Applying  \eqref{eq:lampermdef} to $\eta^1 \subseteq R_{r+t}$,  we thus see that 
$w^1(i)=i$ for all $1\leq i\leq r-t$.  Therefore, the permutation $w:=w^2w^1$ has two descents by Lemma \ref{le:constructfactors}; namely $\D(w)=\{r-t,r+t\}$. For the given partitions $\lambda,\mu \subseteq R_r$, denote by $u:=\varphi_r(\lambda)$ and $v:=\varphi_r(\mu)$, and note that $\D(u)=\D(v) = \{r\}$ by Remark \ref{re:thetainverse}.

We aim to apply Theorem \ref{thm:lstodarephrase}, which requires verifying a hypothesis on the relationship among these descent vectors.  Defining $\t:=(0^{r-t},1,\ldots,t-1,t,t-1,\ldots,1,0^{m-t})$, recall from \eqref{eq:tilded} that 
$\tilde \t = \sum_{i \in [k]} t_i(\varepsilon_{i-1} -2\varepsilon_i + \varepsilon_{i+1}),$ where $t_i$ denotes the $i^{\text{th}}$ entry of $\t$.  Due to the palindromic nature of $\t$, we may compute directly that $\tilde \t = \varepsilon_{r-t} - 2\varepsilon_r + \varepsilon_{r+t}$.  In addition, we have shown above that $\D(w)-\D(u)-\D(v) = (\varepsilon_{r-t} + \varepsilon_{r+t}) - \varepsilon_r -\varepsilon_r$, in which case $\tilde \t = \D(w)-\D(u)-\D(v)$.
Therefore, Theorem~\ref{thm:lstodarephrase} and Remark \ref{re:thetainverse} imply that
\begin{equation*}
c_{u,v}^{w,\t} = 
c_{\varphi_r(\lambda),\varphi_r(\mu)}^{w,\t}= C_{\tilde \lambda^{\downarrow}(\varphi_r(\lambda)),\tilde \lambda^{\downarrow}(\varphi_r(\mu))}^{\tilde \lambda^{\downarrow}_w,(k)}=  C_{\lambda,\mu}^{\tilde \lambda^{\downarrow}_w,(k)}\,.
\end{equation*}

It thus remains to prove that $\tilde \lambda^{\downarrow}_w = {\nu\oplus d}$.
Since $w^1 \in S_n^{r+t}$ and $w^2 \in S_n^{r-t}$ satisfy the conditions of Proposition \ref{prop:inversions}, we have that 
\begin{equation*}
\left( \tilde \lambda_w^{\downarrow}\right) '=\left( \tilde \lambda_{w^2}^{\downarrow}\right) '+\left( \tilde \lambda_{w^1}^{\downarrow}\right) '\, .
\end{equation*}
Since $w^1=\varphi_{r+t}(\eta^1)$ and $w^2=\varphi_{r-t}(\eta^2)$ by definition, applying the inverse relationship from Remark \ref{re:thetainverse} then gives us 
\begin{equation}\label{eq:trsum}
\left( \tilde \lambda_w^{\downarrow}\right) ' = (\eta^2)' + (\eta^1)'.
\end{equation}
Since the sum of partitions is taken coordinate-wise on parts, then in fact $ (\eta^2)' + (\eta^1)' = (\eta^1, \eta^2)'$. 
 Transposing \eqref{eq:trsum} thus yields
\[
 \tilde \lambda_w^{\downarrow}= (\eta^1, \eta^2) = \nu \oplus d,
\]
as required. In particular, given Grassmann permutations $u, v \in S_n^r$, a permutation $w \in S_n$ with two descents in positions $r\pm t$, and a palindromic degree $\t$ so that $\tilde \t = \D(w)-\D(u)-\D(v)$, the correspondence \eqref{eq:LJFlcomb} provides an inverse to $\LJFl^{-1}$ from Theorem \ref{thm:lstodarephrase}.
\end{proof}

To illustrate Theorem \ref{thm:LJFlcomb}, we reconsider the affine Littlewood-Richardson coefficient $C_{\lambda, \mu}^{\nu\oplus d,(k)}$ resulting from the application of $\LJGr$ in Example \ref{ex:LJGrfav}.  

\begin{example}\label{ex:LJFlcombfav}
Recall from Example \ref{ex:LJGrfav} that $n=5$ and $r=2$ with $\lambda = (2,2,1), \mu = (1,1,0), \nu = (2,0,0) \subset R_2$ and $d=1$.  
Adding one $n$-rim hook to $\nu$ with its head in column $r=2$, and overlaying the rectangle $R_2$, we see that $(\nu \oplus 1)^{\vee_2} = (1,0,0)$, and so $t = \diag_0(1,0,0)= 1$. Splitting the partition $\nu \oplus d = (\eta^1, \eta^2)$ such that $\eta^1$ has $3-1=2$ parts, we obtain
$$\nu\oplus 1=
{\tiny\tableau[scY]{\times|\times| \times|\times ,\times|\tf ,\tf}}
\implies
(\eta^1,\eta^2) = \left(
\,
{\tiny \tableau[scY]{
\times,\times|
\tf,\tf|
}}
\;,\;
 {\tiny\tableau[scY]{
\times|\times|
\times|
}}
\,
\right).$$
With $\eta^1 = (2,2)$ and $\eta^2 = (1,1,1)$, viewing $\eta^1 \subset R_{2+1} = R_3$ and applying Remark \ref{re:thetainverse} with $j=3$, we have $\varphi_3(\eta^1) = [12534] \in S^3_5$.  Similarly, viewing $\eta^2 \subset R_1$ and applying Remark \ref{re:thetainverse}, we have $\varphi_1(\eta^2)=[21345] \in S^1_5$.  Taking the product, we obtain $\varphi_1(\eta^2)\varphi_3(\eta^1) = [21534].$

Having already computed that $\varphi_2(\mu) = [25134]$ in Example \ref{ex:thetainverse}, in like manner $\varphi_2(\lambda) = [13245]$ via Remark \ref{re:thetainverse}.   Applying Theorem \ref{thm:LJFlcomb}, where $\t = (0^{2-1}, 1, 0^{3-1}) = (0,1,0,0)$, we then have the following equality
\begin{equation*}
C_{\,\tableau[pcY]{|,|,}\, , \, \tableau[pcY]{||}}^{\tableau[pcY]{\fl| \fl| \fl|\fl,\fl|,} \, , \, (4)} \xlongequal{\LJFl} c_{[13245]\, ,\, [25134]}^{[21534]\, ,\, (0,1,0,0)} .
\end{equation*}
Note that this equality is identical to \eqref{eq:LJFlbackward} from Example \ref{ex:LJFlfav}, illustrating that Theorem \ref{thm:LJFlcomb} provides an inverse to the correspondence $\LJFl^{-1}$ on the image of $\LJGr$.
\end{example}


\section{Proof of the Main Theorem}\label{sec:proof}

This goal of this section is to prove the Main Theorem, by unifying the combinatorics indexing the third permutations resulting from the two different compositions $\LJFl \circ \LJGr$ and $\PC \circ \SD$.  This third indexing permutation originates from $\LJGr$ via addition of rim hooks, and so Section \ref{sec:rimhooks} provides explicit formulas for this partition $\nu \oplus d$.  The correspondence $\LJFl$ then involves a separation of $\nu \oplus d= (\eta_t^1, \eta_t^2)$ into a pair of partitions, which we recognize in Section \ref{sec:parts} as a pair of partitions naturally obtained from the cycled image $\cyc^r(\nu)$ which occurs in $\SD$. In Section  \ref{sec:perms},  we connect the product of permutations   appearing in $\LJFl$  to the element $w_0^{P'_t}$ from $\PC$.  The proof then immediately follows in Section \ref{sec:mainproof}.

\subsection{Rim hooks and pairs of partitions}\label{sec:rimhooks}

We begin by giving an explicit description of the shape $\nu \oplus d$, as well as the 
pair of partitions $\nu\oplus d=(\eta_t^1,\eta_t^2)$.
Recall that the transpose of partition $\nu=(\nu_1,\ldots,\nu_m)\subseteq R_r$ is the partition
$\nu'=(\nu_1',\ldots,\nu_r')$ where $\nu_j'=\#\{ i\in[m] \mid \nu_i\geq j\}$.

\begin{lemma}  \label{lem:nuoplusd}   
For any $\nu = (\nu_1, \dots, \nu_m) \subseteq R_r$ and $d \in \ZZ_{\geq 0}$
with $a_d:=\nu'_{r-d}$, we have
\begin{equation}
\label{eq:nuoplusdparts}
\nu\oplus d = (r^{d+a_d},\nu_{a_d+1}+d,\dots,\nu_m+d,\nu_1-r+d,\dots,\nu_{a_d}-r+d)\,.
\end{equation}
Moreover, separating $\nu\oplus d=(\eta_t^1,\eta_t^2)$  as in Definition \ref{def:etas} with 
$t=\diag_0((\nu\oplus d)^{\vee_r})$, 
  \begin{align*}
\eta_t^1 &= (r^{d+a_d},\nu_{a_d+1}+d,\dots,\nu_{m-\delta}+d), \\
\eta_t^2 &= (\nu_{m-\delta+1}+d,\dots,\nu_m+d,\nu_1+d-r,\dots,\nu_{a_d}+d-r),
\end{align*}
where $\delta:=\diag_0(\nu^\vee)$.
\end{lemma}

\begin{proof}
In order to prove  \eqref{eq:nuoplusdparts}, we proceed by  induction on $d$. For $d=0$, the equality holds since $\nu_i=r$ for $i\in[a_0] = [\nu_r']$.
Assuming \eqref{eq:nuoplusdparts} for $\nu\oplus d$,  adding a rim hook to $\nu\oplus d$ gives
$$
\nu\oplus (d+1) = (r^{d+a_d+1},\nu_{a_d+1}+d+1,\dots,\nu_m+d+1,\nu_1-r+d+1,\dots,\nu_{a_d}-r+d+1).
$$
Let $b=a_{d+1}-a_d=\nu_{r-d-1}'-\nu_{r-d}'$ and observe our claim holds trivially at $b=0$. 
When $b>0$, then $\nu_{r-d-1}'>\nu_{r-d}'$ which implies that 
\begin{equation}\label{eq:rd1values}
\nu_{a_{d+1}}=\nu_{\nu_{r-d-1}'}=r-d-1
=\nu_{a_d+1}=\nu_{a_d+2}=\cdots =\nu_{a_d+b}.
\end{equation}
Therefore, for all $i \in [b]$, we have $\nu_{a_d+i}+d+1 = r$ and so
$$
\nu\oplus (d+1) = (r^{d+a_d+1},r^b, \nu_{a_{d+1}+1}+d+1,\dots, 
\nu_m+d+1,\nu_1-r+d+1,\dots,\nu_{a_{d}}-r+d+1)\,.
$$
We can transform this to precisely the form of identity~\eqref{eq:nuoplusdparts} at $d+1$ using $a_{d+1}=a_d+b$, and by
noting via \eqref{eq:rd1values} that $\nu_{a_d+1}-r+d+1=\cdots=\nu_{a_{d+1}}-r+d+1=0$.

The explicit form~\eqref{eq:nuoplusdparts} for $\nu\oplus d$ immediately
implies expressions for $(\eta^1_t,\eta^2_t)$ since
$\eta^1_t$ consists of the first $m-t$ rows, and $\delta:=\diag_0(\nu^\vee)=d+t$ by Remark~\ref{re:diagdist}. 
\end{proof}

We now illustrate the formulas for $\nu \oplus d=(\eta_t^1, \eta_t^2)$ from Lemma \ref{lem:nuoplusd} for several values of $d$.

\begin{example}\label{ex:nuoplusd}
Let $n=15$ and $\nu=(8,7,5,2,1)\subset R_{10}$, in which case $m=5$.  Here $\delta=\diag_0(\nu^\vee)=3$, and so $m-\delta = 2$.  

For $d=0$, we have $t=\diag_0(\nu^\vee)= 3$ and $a_d=0$.  Equation \eqref{eq:nuoplusdparts} in the case $d+a_d=0$ simply returns $\nu \oplus 0 = (\nu_1, \dots, \nu_m)$, which we separate as
$$
\nu \oplus 0= {\huge\tableau[pcY]{
\tf,,,,,,,,,|
\tf,\tf,,,,,,,,|
\tf ,\tf,\tf,\tf,\tf,,,,,|\tf  ,\tf  ,\tf  ,\tf  ,\tf  ,\tf  ,\tf  ,,,|
 \tf  ,\tf  ,\tf  ,\tf  ,\tf  ,\tf  ,\tf  ,\tf,,|}}
 \quad\mapsto\quad
 (\eta_3^1,\eta_3^2 )=\left(
\,
{\huge\tableau[pcY]{
\tf ,\tf  ,\tf  ,\tf  ,\tf  ,\tf  ,\tf  |
 \tf  ,\tf  ,\tf  ,\tf  ,\tf  ,\tf  ,\tf  ,\tf|}}
\,\;,\,\;
{\huge \tableau[pcY]{ 
\tf|\tf,\tf|\tf,\tf,\tf,\tf,\tf}}
\,
\right).
$$

 For  $d=2$, we have $t=\diag_0((\nu\oplus 2)^\vee)= 1$ and $a_d=1$.  Equation \eqref{eq:nuoplusdparts} then says that $\nu\oplus 2 =(10^{2+1},7+2,5+2,2+2,1+2,8-10+2)= (10,10,10,9,7,4,3,0)$, which is then separated as

$$
\nu \oplus 2 =
{\Tiny\tableau[scY]{
2,2,2|
1,1,2,2|
\tf,1,1,2,2,2,2|
\tf,\tf,1,1,1,1,2,2,2|
\tf,\tf,\tf,\tf,\tf,1,1,1,2,2|
\tf ,\tf,\tf,\tf,\tf,\tf,\tf,1,1,2|
\tf,\tf,\tf,\tf,\tf,\tf,\tf,\tf,1,1|}}
\, 
 \quad \mapsto\quad
(\eta_1^1,\eta_1^2) = \left(
\,
{\Tiny \tableau[scY]{
\tf,\tf,1,1,1,1,2,2,2|
\tf,\tf,\tf,\tf,\tf,1,1,1,2,2|
\tf ,\tf,\tf,\tf,\tf,\tf,\tf,1,1,2|
\tf,\tf,\tf,\tf,\tf,\tf,\tf,\tf,1,1|}}
\;,\;
 {\Tiny\tableau[scY]{
2,2,2|
1,1,2,2|
\tf,1,1,2,2,2,2|
}}
\, 
\right).$$
In the figure above, the boxes containing 1's denote the cells of the first rim hook added, and the second rim hook is denoted by those cells containing 2's.
\end{example}

\subsection{Cycling and pairs of partitions}\label{sec:parts}

In order to prove the Main Theorem, we relate $\nu\oplus d=(\eta_t^1,\eta_t^2)$ to pairs of partitions constructed after applying the cycling map to $\nu$.  Throughout the rest of the paper, we thus denote by $\rho:=\cyc^r(\nu)$ the partition which results from  cycling the first $r$ bits of $b_\nu$ to the end of the string. It is helpful to start by formulating $\rho$ in terms of $\nu$.

\begin{lemma}\label{le:nicerho}
Let $\nu = (\nu_1, \dots, \nu_m) \subseteq R_r$, and denote by  $\rho=\cyc^r(\nu)$. For
$\delta:=\diag_0(\nu^\vee)$, we have
\[
 \rho=  (\nu_{m-\delta+1}+\delta,\dots,\nu_m+\delta,\nu_1-r+\delta,\dots,\nu_{m-\delta}-r+\delta)\,.
\]
Furthermore, $\delta = \diag_0(\rho)$. 
\end{lemma}

\begin{proof}

Recall from Definition \ref{def:etas} that  $\eta_\delta^1$ and $\eta_\delta^2$ are defined as the bottom $m-\delta$ and top $\delta$ rows of $\nu$, respectively, so that
\begin{align}
\eta_\delta^1 &=  (\nu_1, \dots, \nu_{m-\delta}) \label{eq:etadelta} \\
\eta_\delta^2 &=(\nu_{m-\delta+1}, \dots, \nu_m), \nonumber
\end{align}
Write the bit string $b_\nu=b^{r}_\nu\, b^m_\nu$
as the concatenation of the substring $b^r_\nu$ containing the first $r$ bits of $b_\nu$, and the substring $b^m_\nu$
containing the last $m$ bits of $b_\nu$. Defining the southwest corner of $R_r$ as the origin, 
the substring $b_\nu^r$ traces the boundary of $\nu$ from the northwest corner of $R_r$ to position $(r-\delta,m-\delta)$,
since $\delta=\diag_0(\nu^{\vee})$; i.e.~$b_\nu^r$ traces the shape $\eta_\delta^2$.  The substring $b_\nu^m$ then traces the remainder of the boundary of $\nu$ from position $(r-\delta,m-\delta)$ to the southeast corner of $R_r$; i.e.~$b_\nu^m$ traces the shape  $\eta_\delta^1-(r-\delta)^{m-\delta}$.

\begin{figure}[!ht]
\centering
\begin{tikzpicture}[x=.9pt,y=.9pt,yscale=-1,xscale=1]
\draw [ dotted ](182.5,209)-- (100,209)-- (100,133) -- (227.5,133) -- (227.5,209);
\fill [gray!10] (182.5,209) rectangle  (227.5,260);
\fill [gray!10] (100,209) rectangle  (183,183);
\draw[<->] (101, 212) -- node[pos=.5,below] {\Tiny{$r-\delta$}}  (182,212) ;
\draw[<->] (95, 209)  -- node[pos=.5,above,rotate=90]   {\Tiny{$m-\delta$}} (95,183);
\draw[<->,densely dashed ](185, 182)  -- node[pos=.5,below,rotate=90]   {\Tiny{$\delta$}} (185,136);
\draw [<->,densely dashed] (185,135)  -- node [pos=.5, below ] {\Tiny{$\delta$}} (226.5,135); 
\draw [densely dotted]  (182.5,183) -- (182.5,133) ;
\draw  [densely dotted]  (182.5,183) --  (227.5,183) ;
\path [draw, fill=gray!40] (100,183)--  (100,147) -- (118.5,147) -- (118.5,163)  -- (153.5,163)  -- (153.5,173) -- (171.5,173)  -- (171.5,183)  ;
\draw  (171.5,183) -- (183,183);
\node at (112,173) {\Tiny{$\eta_\delta^2$}};
\draw [dotted]    (100,183) -- (182.5,183) ;
\fill [gray!40] (227.5,223) -- (245.5,223) -- (245.5,239)   -- (280.5,239) --  (280.5,249)
 -- (300,249) -- (300,260) -- (227.5,260)-- cycle;
\path[draw]   (227.5,223) -- (245.5,223) -- (245.5,239)   -- (280.5,239) --  (280.5,249)
 -- (300,249) -- (300,260) -- (227.5,260);
 \node at (255,248) {\Tiny{$\rho^{R^\delta}=\eta_\delta^2$}};
\path[draw]   (182.5,183) -- (182.5,191) -- (213.5,191) -- (213.5,202)  -- (227.5,202) 
--(227.5,209) ;
\fill[pattern=north east lines]   (182.5,183) -- (182.5,191) -- (213.5,191) -- (213.5,202)  -- (227.5,202) 
--(227.5,209) -- (182.5,209);
\draw (182.5,183) -- (182.5,191) -- (213.5,191) -- (213.5,202)  -- (227.5,202) 
--(227.5,209);
\draw [dotted] (227.5,209) -- (182.5,209);
\draw  [dash pattern={on 0.84pt off 2.51pt}]  (182.5,183) -- (182.5,209) ;
\draw  [dash pattern={on 0.84pt off 2.51pt}]  (227.5,183) -- (300,183) ;
\draw  [dash pattern={on 0.84pt off 2.51pt}]  (300,183) -- (300,260) ;
\draw    (227.5,209) -- (227.5,223) ;
\draw    [dotted](227.5,223) -- (227.5,260) ;
\draw[<->,densely dashed ](185, 259)  -- node[pos=.5,below,rotate=90]   {\Tiny{$\delta$}} (185,212);
\draw [<->,densely dashed] (186,212)  -- node [pos=.5, below ] {\Tiny{$\delta$}} (226.5,212); 
\draw [densely dotted]  (182.5,260) -- (182.5,209) ;
\draw  [densely dotted]  (182.5,260) --  (227.5,260) ;
\node at (270,170) {\Tiny{$\eta_\delta^1-(r-\delta)^{m-\delta}=\rho^{L^\delta}$}};
\draw[->] (260,180) to [out=150,in=-10, looseness=1]  (200,200);
\node at (126,198) {\tiny{$\nu$}};
\node at (219,250) {\tiny{$\rho$}};
\end{tikzpicture}
\caption{The partition $\nu \subseteq R_r$ and its cycled image $\rho = \cyc^r(\nu)$.}
\label{fig:cyclerho}
\end{figure}

Now cycling to obtain $\rho=\cyc^r(\nu)$, we have $b_\rho= b^m_\nu \,b^r_\nu$.  By our previous observations, $\rho$ can  be constructed by beginning at the northwest corner of $R_r$ and tracing  the partition $\eta_\delta^1-(r-\delta)^{m-\delta}$ to position $(\delta,\delta)$, and then tracing the partition $\eta_\delta^2$ from there to the southeastern corner of $R_r$; see  Figure \ref{fig:cyclerho}.
That is,
\begin{equation}\label{eq:rhoeta}
\rho = \left( \eta_\delta^2+\delta^\delta, \eta_\delta^1-(r-\delta)^{m-\delta}\right).
\end{equation}
From this description,  we observe that $\diag_0(\rho)=\delta$ as claimed.  Finally, comparing \eqref{eq:rhoeta} to \eqref{eq:etadelta}, we obtain the result.
\end{proof}

Given a partition containing a $t\times t$ rectangle, we now define an associated pair of partitions.

\begin{definition}\label{def:rho} 
Given  $\rho \subseteq R_r$ and $t \in \ZZ_{\geq 0}$ such that $(t^t)\subseteq \rho$, let $a:=\rho_t'$
and define a pair of partitions $(\rho^{L^t}, \rho^{R^t})$ by
\begin{enumerate}
\item $\rho^{L^t}= (t^{a-t},\rho_{a+1}, \dots, \rho_{m})$ is the upper lefthand portion of $\rho$ lying above the rectangle $(t^t) \subseteq \rho$, or equivalently, $\left(\rho^{L^t}\right)' =(\rho'_1-t,\dots,\rho'_t-t)$, and
\item $\rho^{R^t} = (\rho_1-t, \dots, \rho_{a}-t)$ is the righthand 
portion of $\rho$ from columns
$t+1,\ldots,r$, or equivalently, $\left(\rho^{R^t}\right)' =(\rho'_{t+1},\dots,\rho'_r)$.
\end{enumerate}
\end{definition}

 We continue by illustrating the construction of $(\rho^{L^t}, \rho^{R^t})$ for several different values of $t$.

\begin{example}\label{ex:rho}
From Example \ref{ex:nuoplusd} with $n=15$ and $\nu=(8,7,5,2,1)\subset R_{10}$,
consider the partition $\rho = \phi^{10}(\nu)$ by first converting $\nu$ to its bit string $b_\nu$. Cycling the first 10 bits (shown below in blue) of $b_\nu$ to the end of the string, we have  
\begin{align*}
 \nu = (8,7,5,2,1) &\  \longleftrightarrow \ b_\nu  = \blue{1010111011} 01011, \\
\rho = (8,5,4,1,0) & \ \longleftrightarrow \ b_\rho =01011\blue{1010111011}.
\end{align*}

Recall from Example \ref{ex:nuoplusd} that when $d=2$, we have $t=\diag_0((\nu\oplus 2)^\vee)= 1$.
Applying Definition \ref{def:rho} with $t=1$, we have $a =\rho_1' = 4$, so that $\rho^{L^1}= (1^{4-1},\rho_{4+1}) = (1,1,1,0)$ and $\rho^{R^1} = (\rho_1-1, \dots, \rho_{4}-1)=(7,4,3,0)$.  We visualize this pair of partitions lying above and to the right of the embedded $1 \times 1$ rectangle as follows:

$$
\rho={\huge
\tableau[pcY]{,,,,,,,,,| \tf ,, ,,,,,,,| \tf ,\tf,\tf,\tf,,,,,,|
\tf ,\tf ,\tf ,\tf,\tf,,,,,| \fl ,\tf ,\tf ,\tf,\tf,\tf,\tf,\tf,,| }}
\qquad 
 \mapsto
\qquad
 (\rho^{L^1}, \rho^{R^1} )=
\left(
\,
{\huge \tableau[pcY]{\tf|\tf |\tf }}
\,\;,\,\;
{\huge \tableau[pcY]{ 
\tf, \tf,\tf|
\tf,\tf,\tf,\tf|
\tf,\tf,\tf,\tf,\tf,\tf,\tf}}
\,
\right)
\, .
$$

Recall from Example \ref{ex:nuoplusd} that when $d=0$, we have $t=\diag_0((\nu\oplus 0)^\vee)= 3$.
Applying Definition \ref{def:rho} with $t=3$, we have $a = \rho_3' = 3$ , and so $a-t=0$.  Therefore, $\rho^{L^3}= (\emptyset, \rho_{3+1},\rho_{5}) = (1,0)$ and $\rho^{R^3} = (\rho_1-3, \dots, \rho_{3}-3)=(5,2,1)$, which we visualize as follows:

$$
\rho=
{\huge
\tableau[pcY]{,,,,,,,,,| \tf ,, ,,,,,,,| \fl ,\fl ,\fl ,\tf,,,,,,|
\fl ,\fl ,\fl ,\tf,\tf,,,,,| \fl ,\fl ,\fl ,\tf,\tf,\tf,\tf,\tf,,| }}
\qquad 
\mapsto 
\qquad
 (\rho^{L^3}, \rho^{R^3} )=
\left(
\,
{\huge \tableau[pcY]{\tf}}
\,\;,\,\;
{\huge \tableau[pcY]{ 
\tf|\tf,\tf|\tf,\tf,\tf,\tf,\tf}}
\,
\right)
\, .
$$
\end{example}

We pause to make one key observation from Lemma~\ref{le:nicerho}.  Given any $\nu\subseteq R_r$ with cycled partition $\rho:=\cyc^r(\nu)$, for any $d \in \ZZ_{\geq 0}$, let
 $t:=\diag_0((\nu\oplus d)^{\vee})$. By  Remark~\ref{re:diagdist},  
we have $0 \leq t\leq \delta = \diag_0(\nu^\vee)$.  Since $\delta = \diag_0(\rho)$ by Lemma \ref{le:nicerho}, for any such $t$ we also have $(t^t) \subseteq \rho$. Thus 
 the pair of partitions $(\rho^{L^t}, \rho^{R^t})$ is defined for $\rho$ with this value of $t$, and is related explicitly to the earlier pair of partitions $(\eta^1_t,\eta^2_t)$ obtained from adding rim hooks as follows.

\begin{prop}\label{prop:rho-nu}
Let $\nu \subseteq R_r$ and $\rho=\cyc^r(\nu)$. Let $d \in \ZZ_{\geq 0}$ and $t=\diag_0((\nu \oplus d)^{\vee_r})$. 
Then,
\begin{equation}\label{eq:etaRL}
(\eta_t^1, \eta_t^2 ) = (\rho^{L^t}+(r-t)^{m-t}, \rho^{R^t}),
\end{equation}
where the partition $\nu \oplus d = (\eta^1_t,\eta^2_t)$ is separated as in Definition \ref{def:etas}.
 \end{prop}

\begin{proof}
For $\nu=(\nu_1, \dots, \nu_m) \subseteq R_r$ and 
$\delta=\diag_0(\nu^\vee)=\diag_0(\rho)$, Lemma~\ref{le:nicerho} gives that
\begin{equation}\label{eq:explicit-rho-pf}
 \rho=  (\nu_{m-\delta+1}+\delta,\dots,\nu_m+\delta,\nu_1-r+\delta,\dots,\nu_{m-\delta}-r+\delta)\,.
\end{equation}

Since $0 \leq t\leq \delta$ by Remark~\ref{re:diagdist},
we have that $a:=\rho_t'\geq \rho_{\delta}'$.
In turn, $\rho_{\delta}'\geq\delta$ since $(\delta^\delta)\subseteq \rho$.
In particular, since $a \geq \delta$, writing $\rho = (\rho_1, \dots, \rho_m)$ gives us
\begin{equation*}
(\rho_1, \dots, \rho_a) = (\nu_{m-\delta+1}+\delta,\dots,\nu_m+\delta,\nu_1-r+\delta,\dots,\nu_{a-\delta}-r+\delta).
\end{equation*}
We can thus write 
\begin{align}
  \rho^{R^t}&= ( \nu_{m-\delta+1}+\delta-t,\dots,\nu_m+\delta-t,\nu_1-r+\delta-t,\dots,\nu_{a-\delta}-r+\delta-t) \label{eq:eta2-rho-t} \\
& =  ( \nu_{m-\delta+1}+d,\dots,\nu_m+d,\nu_1+d-r,\dots,\nu_{a-\delta}+d-r),\nonumber
\end{align}
using that $\delta=d+t$ by Remark~\ref{re:diagdist}. 

 Having identified the first $a$ parts of $\rho$ to construct $\rho^{R^t}$ above, using \eqref{eq:explicit-rho-pf} we also have
\begin{equation*}
\rho^{L^t}   =
 (t^{a-t},\nu_{a-\delta+1}-r+\delta, \dots, \nu_{m-\delta}-r+\delta). 
\end{equation*}
Since $\rho^{L^t}$ has exactly $m-t$ parts by definition, we compute that
\begin{align}
\rho^{L^t}+(r-t)^{m-t}&= ( r^{a-t},\nu_{a-\delta+1}+\delta-t, \dots, \nu_{m-\delta}+\delta-t) \label{eq:eta1-rho-t} \\
& =  ( r^{a-\delta+d},\nu_{a-\delta+1}+d, \dots, \nu_{m-\delta}+d), \nonumber
\end{align}
where again we have used the relation $\delta=d+t$. 

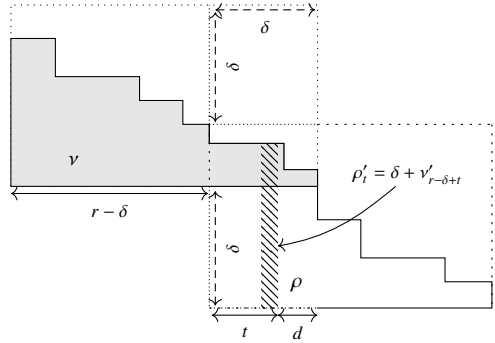
\begin{figure}[!htb]
 \centering
\begin{tikzpicture}[x=.9pt,y=.9pt,yscale=-1,xscale=1]
\draw [ dotted ](182.5,209)-- (100,209)-- (100,133) -- (227.5,133) -- (227.5,209);
\draw[<->] (100, 212) -- node[pos=.5,below] {\Tiny{$r-\delta$}}  (182,212) ;
\draw[<->] (183.5, 263) -- node[pos=.5,below] {\Tiny{$t$}}  (211,263) ;
\draw[<->] (211, 263) -- node[pos=.5,below] {\Tiny{$d$}}  (227.5,263) ;
\draw[<->,densely dashed ](185, 182)  -- node[pos=.5,below,rotate=90]   {\Tiny{$\delta$}} (185,136);
\draw [<->,densely dashed] (185,135)  -- node [pos=.5, below ] {\Tiny{$\delta$}} (226.5,135); 
\draw [densely dotted]  (182.5,183) -- (182.5,133) ;
\draw  [densely dotted]  (182.5,183) --  (227.5,183) ;
\path [draw, fill=gray!20]   (100,147) --  (118.5,147) -- (118.5,163)  -- (153.5,163)  -- (153.5,173) -- (171.5,173)  -- (171.5,183)
 --   (182.5,183) -- (182.5,191) -- (213.5,191) -- (213.5,202)  -- (227.5,202) 
--(227.5,209) -- (182.5,209)--  (100,209) -- cycle ;
\path[draw]   (227.5,209) --  (227.5,223) -- (245.5,223) -- (245.5,239)   -- (280.5,239) --  (280.5,249) -- (300,249) -- (300,260) -- (227.5,260);
\path[draw, fill=gray!20]   (182.5,183) -- (182.5,191) -- (213.5,191) -- (213.5,202)  -- (227.5,202) 
--(227.5,209) -- (182.5,209) ;
\draw  [dash pattern={on 0.84pt off 2.51pt}]  (182.5,183) -- (182.5,209) ;
\draw  [dash pattern={on 0.84pt off 2.51pt}]  (182.5,260) -- (227.5,260) ;
\draw  [dash pattern={on 0.84pt off 2.51pt}]  (227.5,183) -- (300,183) ;
\draw  [dash pattern={on 0.84pt off 2.51pt}]  (300,183) -- (300,260) ;
\draw [densely dotted]  (182.5,209) --  (182.5,260) --  (227.5,260)  ;
\fill [pattern=north west lines]   (204,191) rectangle (211,260);
\node at (265,204) {\Tiny{$\rho'_t=\delta+\nu'_{r-\delta+t}$}};
\draw[->] (260,209) to [out=140,in=-10, looseness=1]  (211,234);
\draw[<->,densely dashed ](185, 259)  -- node[pos=.5,below,rotate=90]   {\Tiny{$\delta$}} (185,211);
\node at (126,198) {\tiny{$\nu$}};
\node at (219,250) {\tiny{$\rho$}};
\end{tikzpicture}
 \caption{The columns of $\nu\subseteq R_r$ and its cycled image $\rho = \cyc^r(\nu)$.}
\label{fig:rho-nu-a}
\end{figure}

 Comparing \eqref{eq:eta1-rho-t} and  \eqref{eq:eta2-rho-t} to the respective formulas for $\eta_t^1$ and $\eta_t^2$ in Lemma \ref{lem:nuoplusd}, we now see that \eqref{eq:etaRL} directly follows, provided that $a_d:=\nu_{r-d}'=a-\delta$. 
 Figure \ref{fig:rho-nu-a} illustrates the relationship between the columns of $\nu$ and $\rho$. In particular, note that   the  $t$-th column of $\rho$ and the  $(r-\delta+t)$-th column of $\nu$ are aligned.  
Thus by comparing their heights, we obtain $\rho_{t}'=\delta + \nu'_{r-\delta+t}$.
Therefore,  using $\delta=d+t$, we have $a_d = \nu'_{r-d}=\nu_{r-\delta+t}'=\rho'_t-\delta=a-\delta$.
\end{proof}

We continue Example \ref{ex:rho} below to illustrate the statement of Proposition \ref{prop:rho-nu}.

\begin{example}\label{ex:etaRLlemma}
 Recall that in our running example we have $n=15$ and $\nu=(8,7,5,2,1)\subset R_{10}$, in which case $\rho=\cyc^{10}(\nu)=(8,5,4,1,0)$. Comparing Examples \ref{ex:nuoplusd} and \ref{ex:rho}, we see that when $t=3$, equivalently $d=0$, we have
 \[(\eta_3^1, \eta_3^2 ) = ((8,7),(5,2,1)) =  ((1,0)+(7,7),(5,2,1))= (\rho^{L^3}+(7,7), \rho^{R^3}).\]
In the case $t=1$, equivalently $d=2$, again from Examples \ref{ex:nuoplusd} and \ref{ex:rho} we have
\[(\eta_1^1, \eta_1^2 ) = ((10,10,10,9),(7,4,3,0)) = (\rho^{L^1}+(9,9,9,9), \rho^{R^1}),\]
demonstrating the statement of Proposition \ref{prop:rho-nu} for two different values of $t$, equivalently $d$.
\end{example}

\subsection{Cycling and pairs of permutations}\label{sec:perms}

 Our final step in the proof of the Main Theorem is to then express the third permutation in the correspondence $\LJFl$ from Theorem \ref{thm:LJFlcomb}  in a form which more naturally relates to the Peterson comparison $\PC$.

\begin{prop}\label{prop:wP'eta}
Let $\nu \subseteq R_r$ and $\rho=\cyc^r(\nu)$. Let $d \in \ZZ_{\geq 0}$, and write $\nu \oplus d=(\eta_t^1,\eta_t^2)$  with  $t=\diag_0((\nu \oplus d)^{\vee_r})$. Then
\[  \varphi_{r-t}(\eta_t^2)\varphi_{r+t}(\eta_t^1) = w_{\rho} w_0^{P_t'}w_0,
\]
where $w_0^{P'_t}$ is the permutation defined in \eqref{eq:w0P'} with $d=t$.
\end{prop}

In the proof of the Main Theorem, Proposition \ref{prop:wP'eta} is the most important result from this section.  
We first demonstrate the statement of Proposition \ref{prop:wP'eta} for a pair of partitions arising from Example \ref{ex:etaRLlemma} , since this verification also illustrates the method of proof.

\begin{example}\label{Ex:wP'bitlemma}
Let $n=15$, and consider the partition $\rho=(8,5,4,1,0) \subset R_{10}$ with $t=3$. 
We illustrate Proposition \ref{prop:wP'eta} in the same manner in which the proof of the general case proceeds. 

Applying Theorem \ref{thm:comparison} with $d=t=3$, we have
$$w_0^{P_3'} = [ 2, 1\ \textcolor{red}{|}\ 5,4,3\ \textcolor{blue}{|}\ 8,7,6\ \textcolor{red}{|}\ 15, 14, 13, 12, 11, 10, 9],$$
where we have marked position $m=5$ in blue and positions $m\pm t = 5\pm 3$ in red.  
Right multiplication by $w_0$ then records the one-line notation for $w_0^{P_3'}$ in reverse order:
$$w_0^{P'_3}w_0 = [ 9,10,11,12,13,14,15\ \textcolor{red}{|}\ 6,7,8\ \textcolor{blue}{|}\ 3,4,5\ \textcolor{red}{|}\ 1,2].$$

Using the bijection between partitions and Grassmannian permutations from \eqref{eq:lampermdef}, we have 
$$w_\rho = [1,3\ \textcolor{red}{|}\ 7,9,13\ \textcolor{blue}{|}\ 2,4,5\ \textcolor{red}{|}\ 6,8,10,11,12,14,15] \in S^5_{15}.$$
Note here that we also artificially view $w_\rho$ as having four ``bins'', created by the same separations in positions $m$ and $m\pm 3$. The effect of right multiplying $w_\rho$ by the permutation $w_0^{P'_3}w_0$ is thus simply to reverse the ordering of these four bins in $w_\rho$; that is, 
$$w_\rho w_0^{P'_3} w_0=[6,8,10,11,12,14,15\ \textcolor{red}{|}\ 2,4,5\ \textcolor{blue}{|}\ 7,9,13\ \textcolor{red}{|}\ 1,3],$$
where the four bins in this product now occur in positions $r$ and $r\pm t$. 

To compare this resulting product to the right-hand side, recall from Example \ref{ex:etaRLlemma} that $\eta_3^1= (8,7)$ and $ \eta_3^2= (5,2,1)$. Then applying  $\varphi_{10\pm 3}$ as in \eqref{eq:thetaGrass}
gives the following correspondences between permutations and partitions: $$
\varphi_{10-3}(\eta_3^2) \longleftrightarrow \left((\eta_3^2)^{\vee_{7}}\right)' =
{\huge \tableau[pcY]{
,,,,|,,,,,|,,,,,,|,,,,,,|,,,,,,|,,,,,,,|,,,,,,, }} = (8,8,7,7,7,6,5) \, ,$$
$$\varphi_{10+3}(\eta_3^1) \longleftrightarrow \left((\eta_3^1)^{\vee_{13}}\right)' =
{\huge \tableau[pcY]{
|,|,|,|,|,}} = (2,2,2,2,2,1,0^7)\, .
$$
From \eqref{eq:lampermdef}, the one-line notation of these permutations can be written as
 \begin{align*}
 \varphi_{7}(\eta_3^2) &= [6,8,10,11,12,14,15\ \textcolor{kellygreen}{|}\ 1,2,3,4,5\ \textcolor{kellygreen}{|}\ 7,9,13]\\
\varphi_{13}(\eta_3^1) &=[1,2,3,4,5,6,7\ \textcolor{red}{|}\ 9,11,12\ \textcolor{blue}{|}\ 13,14,15\ \textcolor{red}{|}\ 8,10] \, ,
\end{align*}
where we have again artificially separated each permutation into several ``bins''. 
Composing these two permutations then gives
\begin{align*}
 \varphi_{7}(\eta_3^2)\varphi_{13}(\eta_3^1)& =[6,8,10,11,12,14,15\ \textcolor{red}{|}\ 2,4,5\ \textcolor{blue}{|}\ 7,9,13\ \textcolor{red}{|}\ 1,3] = w_\rho w_0^{P'_3} w_0,
\end{align*}
confirming Proposition \ref{prop:wP'eta} in this example.

\end{example}

The proof of Proposition \ref{prop:wP'eta} follows the same approach as the calculations in Example \ref{Ex:wP'bitlemma}.

\begin{proof}[Proof of Proposition \ref{prop:wP'eta}]

Write the permutation $w_\rho = [w_1 \cdots  w_n]$ in one-line notation, where recall that $\rho=(\rho_1,\dots, \rho_m) \subseteq R_r$, so that $w_\rho \in \mGrass$.  We will start by dividing $w_\rho$ into four ``bins,'' which we denote by $B_1, B_2, B_3, B_4$, created by placing a separation at the location of the descent $m$, as well as two artificial separations at $m\pm t$ as follows:
$$w_\rho = [w_1\ \cdots \ w_{m-t} \ \textcolor{red}{|}\  w_{m-t+1} \ \cdots \ w_{m} \ \textcolor{blue}{|}\  w_{m+1} \ \cdots \ w_{m+t} \ \textcolor{red}{|}\  w_{m+t+1}\ \cdots \ w_n] =: [B_1\ B_2\ B_3\ B_4].$$
In some degenerate cases the number of distinct bins could be strictly less than four, of course, but we provide the details for the generic case in which there are four bins since the other cases follow by the same argument.

We now calculate the permutation $w_\rho w_0^{P'_t}w_0$.  Recall from Theorem \ref{thm:comparison} with $d=t$ that
$$w_0^{P_t'} = [ m-t \cdots 1\ \textcolor{red}{|}\ m \cdots m-t+1\ \textcolor{blue}{|}\ m+t \cdots m+1\ \textcolor{red}{|}\ n \cdots m+t+1],$$
where the separations occur in positions $m$ and $m\pm t$ as above.
The effect of right multiplication by $w_0$ is to record the one-line notation for $w_0^{P_t'}$ in reverse order, and so
$$w_0^{P_t'}w_0 = [ m+t+1 \cdots n\ \textcolor{red}{|}\ m+1 \cdots m+t\ \textcolor{blue}{|}\ m-t+1 \cdots m\ \textcolor{red}{|}\ 1 \cdots m-t],$$
where now the separations occur in positions $r$ and $r\pm t$.
The composition $w_\rho w_0^{P'_t}w_0$ thus simply reverses the ordering of the four bins in $w_\rho$, meaning that
$$w_\rho w_0^{P'}w_0 = [B_4\ B_3\ B_2\ B_1].$$

It thus suffices to prove that the one-line notation 
for 
$$ \varphi_{r-t}(\eta_t^2)\varphi_{r+t}(\eta_t^1) =  [B_4\ B_3\ B_2\ B_1]\,.
$$   
To this end, since $\phi^r(\nu)=\rho=(\rho_1,\ldots,\rho_m)\subseteq R_r$, 
we use Lemma \ref{le:comtra} to express the bins in terms of $\rho$ with
$w_\rho= [\rho_{m}+1 \cdots \rho_1+m \mid m+1-\rho_1' \cdots n-\rho_{r}'] = [B_1 \,B_2 \,B_3 \,B_4]$.
In particular,
\begin{align*}
B_1&=[\rho_m+1 \cdots \rho_{t+1}+m-t], 
&B_2&= [\rho_t+m-t+1 \cdots \rho_1+m], \\
B_3&=[ m+1-\rho_1' \cdots m+t-\rho_t'], 
&B_4&=[m+t+1-\rho_{t+1}' \cdots n-\rho_r'].
\end{align*}

To determine the product $\varphi_{r+t}(\eta_t^2) \varphi_{r+t}(\eta_t^1) $,  we first study
the two permutations separately.
Recall that $\varphi_{r-t}(\eta_t^2)$ is the permutation corresponding to $\left((\eta_t^2)^{\vee_{r-t}}\right)'$.
For the pair of permutations  $(\rho^L,\rho^R):=(\rho^{L^t},\rho^{R^t})$ described by Definition \ref{def:rho},  viewing $\eta_t^2\subseteq R_{r-t}$, by  Proposition \ref{prop:rho-nu}, we have
\begin{align*}
\eta_t^2&=\rho^{R}=(\rho_1-t,\ldots,\rho_a-t,0^{m-a+t})\subseteq R_{r-t},\\
(\eta_t^2)'&=(\rho^R)' =(\rho'_{t+1},\ldots,\rho_r')\subseteq R_{m+t},
\end{align*}
where $a:= \rho_t'$. Therefore, 
Lemma~\ref{le:comtra} applied to $\eta_t^2$ with $j=m+t$ gives the description
$$
\varphi_{r-t}(\eta_t^2) = [m+t+1-\rho'_{t+1} \cdots  n-\rho_{r}' \mid 1 \cdots m-a+t\mid \rho_a+m-a+1  \cdots \rho_1+m],$$
or equivalently, since $(t^t)\subseteq \rho$ implies that $a=\rho'_t\geq t$, 
\begin{equation}\label{eq:wb1}
[B_4\mid 1\cdots  m-a+t\mid \cdots \mid B_2]\, .
\end{equation}

A similar study applies to $\varphi_{r+t}(\eta_t^1)$, the permutation corresponding to $\left((\eta_t^1)^{\vee_{r-t}}\right)'$. By Proposition \ref{prop:rho-nu} and Definition \ref{def:rho},  we have
\begin{align*}
\eta_t^1&=\rho^{L^t}+(r-t)^{m-t} =  (r^{a-t},r-t+\rho_{a+1},\ldots,r-t+\rho_m)\subseteq R_{r+t},\\
(\eta_t^1)'&=(\rho^{L^t}+(r-t)^{m-t})' =((m-t)^{r-t},\rho_1'-t,\ldots,\rho_t'-t,0^t)\subseteq R_{m-t}.
\end{align*}
Therefore, by Lemma~\ref{le:comtra} applied to $\eta_t^1$ with  $j=m-t$, we have
\begin{align}
\varphi_{r+t}(\eta_t^1)&= [1 \cdots r-t\mid n-t+1-\rho_1' \cdots n-\rho_t'\mid
n-t+1 \cdots n\mid \cdots]  \nonumber \\
&= [1 \cdots r-t \mid B_3 +(r-t) \mid n-t+1 \cdots n \mid \cdots],
\label{eq:wb2}
\end{align}
where $B_3+(r-t)$ denotes the string obtained by adding $r-t$ to each entry of $B_3$.

We can see immediately from \eqref{eq:wb1} and \eqref{eq:wb2} that 
 \begin{equation}\label{eq:varphiproduct}
 \varphi_{r-t}(\eta_t^2)\varphi_{r+t}(\eta_t^1)= [ B_4\mid \cdots \mid B_2 \mid \cdots]\,.
 \end{equation}
  To compute the values between $B_4$ and $B_2$ in the product, we must identify the image of $\{r-t+1, \dots, r\}$ under the composition $\varphi_{r-t}(\eta_t^2)\varphi_{r+t}(\eta_t^1)$.  
Since the entries of $B_3$ are increasing with maximum entry $w_{m+t}=m+t-\rho_t'=m+t-a$,  the values $w_{m+j}$ of bin $B_3$ satisfy $1 \leq w_{m+j} \leq m-a+t$, for $1 \leq j \leq t$.
Now considering the image of $r-t+i$ for $1 \leq i \leq t$ under $\varphi_{r+t}(\eta_t^1)$, we have $\varphi_{r+t}(\eta_t^1)(r-t+i)= w_{m+i}+(r-t)$ by \eqref{eq:wb2}, and in particular, 
\[1+r-t \leq \varphi_{r+t}(\eta_t^1)(r-t+i) \leq (m-a+t)+(r-t)=n-a\] 
for all $1 \leq i \leq t.$ 
From \eqref{eq:wb1}, note that $\varphi_{r-t}(\eta_t^2)(j) = j-(r-t)$ for all $r-t+1 \leq j \leq n-a$, since $B_4$ has $r-t$ entries.  We thus obtain, for any $1 \leq i \leq t$, 
\[ \varphi_{r-t}(\eta_t^2)\varphi_{r+t}(\eta_t^1)(r-t+i) = \left(w_{m+i}+(r-t)\right) -(r-t) = w_{m+i}.\]
In other words, the entries in the product \eqref{eq:varphiproduct} which occur between $B_4$ and $B_2$ are given precisely by $B_3$; that is, $\varphi_{r-t}(\eta_t^2)\varphi_{r+t}(\eta_t^1) = [B_4 \, B_3 \, B_2 \, \cdots]$.

Finally, since  $\varphi_{r+t}(\eta_t^1) \in S^{r+t}_m$ and $\varphi_{r-t}(\eta_t^2) \in S^{r-t}_m$ and $\varphi_{r+t}(\eta_t^1)(i) = i$ for all $i \in [r-t]$,  then Lemma \ref{le:constructfactors} says that $D(\varphi_{r-t}(\eta_t^2)\varphi_{r+t}(\eta_t^1)) = \{r-t,r+t\}$. Note, however, that $D\left( [B_4 \, B_3 \, B_2 \, \cdots] \right)\subseteq \{r\pm t\}$.  
Since $[n]$ is the disjoint union of the values in $B_1, B_2, B_3, B_4$, the remaining entries of $\varphi_{r-t}(\eta_t^2)\varphi_{r+t}(\eta_t^1)$ necessarily consist of precisely the entries of the ordered sequence $B_1$.  Therefore,  $\varphi_{r-t}(\eta_t^2)\varphi_{r+t}(\eta_t^1) = [B_4\  B_3\  B_2\  B_1] = w_\rho w_0^{P'_t}w_0$, as required.
\end{proof}

Before proceeding to the proof of the Main Theorem in the final subsection, we conclude by illustrating the  key propositions from this section, as they apply to our more complete running example from Section \ref{sec:background}.

\begin{example} 
From Example \ref{ex:LJFlcombfav}, with $n=5$ and $r=2$ so that $m=3$, we have $\nu=(2,0,0) \subset R_2$, and $d=1$. 
Further recall that $t=1$ and $\eta^1=(2,2) \subset R_{2+1}$ and $\eta^2=(1,1,1) \subset R_{2-1}$, and from Example \ref{ex:LJFlcombfav}, we have $\varphi_{2-1}(\eta^2)=[2 1 3 4 5]$ and $\varphi_{2+1}(\eta^1)= [1 2 5 3 4]$.
Now define $\rho=\cyc^2(\nu)=(2,2,2)$ so that $\rho^{L^1}=(1,1)$ and $\rho^{R^1}=(1,1,1)$ by Definition \ref{def:rho}.
We thus see that $\eta^1=\rho^{L^1}+(1,1)$ and $\eta^2=\rho^{R^1}$, illustrating Proposition \ref{prop:rho-nu}. 
Finally, $w_0^{P'_1}=[2 1 3 4 5]$ so that $w^{P'_1}_0w_0=[5 4 3 1 2]$.  Since $w_\rho = [34512]$, we then see that
\[ \varphi_{2-1}(\eta^2)\varphi_{2+1}(\eta^1) = [2 1 5 3 4] =w_\rho w_0^{P'_1}w_0\, , \]
demonstrating Proposition \ref{prop:wP'eta}, the crucial step equating the third permutations indexing the output Littlewood-Richardson coefficients from Example \ref{ex:FlTfav} and Example \ref{ex:LJFlcombfav}.
\end{example}

\subsection{Proof of the Main Theorem}\label{sec:mainproof}

The proof of the Main Theorem now follows by applying each of the five comparisons in the order suggested by the diagram in the statement, plus Proposition \ref{prop:wP'eta} and two direct applications of the correspondence \eqref{eq:thetaGrass}.

\begin{proof}[Proof of the Main Theorem]
Given any partitions $\lambda, \mu, \nu \subseteq R_r$ and any integer $d \in \ZZ_{\geq0}$ such that $|\lambda| + |\mu| = |\nu| + nd$, combining the quantum-to-affine correspondences from Theorem \ref{thm:LucJenGr} and Theorem \ref{thm:LJFlcomb}, we have
\begin{equation}\label{eq:LJGrLJFl}
c_{\lambda,\mu}^{\nu,d}\xlongequal{\LJGr}
C_{\lambda, \mu}^{\nu\oplus d, (k)}
\xlongequal{\LJFl} c_{\varphi_r(\lambda),\varphi_r(\mu)}^{\varphi_{r-t}(\eta^2)\varphi_{r+t}(\eta^1),\t}
\end{equation}
where $t = \diag_0((\nu\oplus d)^{\vee})$ and $\t=(0^{r-t},1,\ldots,t-1,t,t-1,\ldots,1,0^{m-t})$.

On the other hand, combining strange duality from Theorem \ref{thm:strangeduality} and the Peterson comparison formula from  Theorem \ref{thm:comparison}, we have
\[ c_{\lambda, \mu}^{\nu, d} \xlongequal{\SD}
c_{\lambda^\vee,{\mu^\vee}}^{\cyc^r(\nu)^\vee,t}
\xlongequal{\PC} c_{w_{\lambda^\vee},w_{\mu^\vee}}^{w_{\cyc^r(\nu)^\vee} w_0^P w_0^{P_t'},\t'}
\] where $t=\diag_0(\nu^\vee)-d=\diag_0((\nu\oplus d)^{\vee})$ and $\t' = (0^{m-t}, 1, 2, \dots, t-1, t, t-1, \dots, 2, 1, 0^{r-t})$. Note in particular that $t'_i=t_{n-i}$ for all $i \in [k]$. Applying the flag transpose $\FlT$ as in  \eqref{eq:FlT}, we thus have
\begin{equation}\label{eq:SDPCT}
\FlT \circ \PC \circ \SD \left( c_{\lambda, \mu}^{\nu, d}\right) = c_{w'_{\lambda^\vee},w'_{\mu^\vee}}^{w_0w_{\cyc^r(\nu)^\vee} w_0^P w_0^{P_t'}w_0,\t} = c_{w'_{\lambda^\vee},w'_{\mu^\vee}}^{ w_{\cyc^r(\nu)}w_0^{P_t'}w_0,\t},
\end{equation}
where we have used the fact that $w_{\cyc^r(\nu)^\vee} =w_0w_{\cyc^r(\nu)}w_0^P$ to simplify the third permutation.

Comparing the resulting quantum Littlewood-Richardson coefficients for $\QH^*(\Fl_n)$ in Equations \eqref{eq:LJGrLJFl} and \eqref{eq:SDPCT}, the first two permutations are equal using the correspondence in \eqref{eq:thetaGrass}, which says that $\varphi_r(\lambda) =w'_{\lambda^\vee}$ and $\varphi_r(\mu) =w'_{\mu^\vee}$.  The  Main Theorem thus follows immediately by applying Proposition \ref{prop:wP'eta} to the third permutations, since $\rho = \cyc^r(\nu)$.
\end{proof}

\renewcommand{\refname}{References}
\bibliography{LizLindaJenniferRefs}
\bibliographystyle{alpha}

\end{document}